\documentclass[11pt]{amsart}

\pdfoutput=1

\usepackage[text={420pt,660pt},centering]{geometry}

\usepackage{cancel}
\usepackage{esint,amssymb} 
\usepackage{graphicx}

\usepackage{MnSymbol}
\usepackage{mathtools} 
\usepackage[colorlinks=true, pdfstartview=FitV, linkcolor=blue, citecolor=blue, urlcolor=blue,pagebackref=false]{hyperref}
\usepackage{microtype}
\usepackage{bm}
\usepackage{mathrsfs}
\usepackage{xcolor}
\usepackage{accents} \usepackage{enumitem}
\setlist{leftmargin=0.5\leftmargin}

\newtheorem{proposition}{Proposition}
\newtheorem{theorem}[proposition]{Theorem}
\newtheorem{lemma}[proposition]{Lemma}
\newtheorem{corollary}[proposition]{Corollary}

\theoremstyle{remark}
\newtheorem{remark}[proposition]{Remark}

\theoremstyle{definition}
\newtheorem{definition}[proposition]{Definition}

\numberwithin{equation}{section}
\numberwithin{proposition}{section}
\numberwithin{figure}{section}
\numberwithin{table}{section}

\newcommand{\Z}{\mathbb{Z}}
\newcommand{\N}{\mathbb{N}}

\newcommand{\R}{\mathbb{R}}

\newcommand{\E}{\mathbb{E}}
\renewcommand{\P}{\mathbb{P}}

\newcommand{\ep}{\varepsilon}
\newcommand{\eps}{\varepsilon}

\renewcommand{\le}{\leqslant}
\renewcommand{\ge}{\geqslant}
\renewcommand{\leq}{\leqslant}
\renewcommand{\geq}{\geqslant}

\renewcommand{\subset}{\subseteq}
\renewcommand{\bar}{\overline}

\newcommand{\Ll}{\left}
\newcommand{\Rr}{\right}
\renewcommand{\d}{\mathrm{d}}
\newcommand{\dr}{\partial}

\newcommand{\mcl}{\mathcal}

\newcommand{\de}{\delta}
\newcommand{\si}{\sigma}

\DeclareMathOperator{\tr}{tr}
\DeclareMathOperator{\supp}{supp}

\newcommand{\la}{\left\langle}
\newcommand{\ra}{\right\rangle}

\newcommand{\D}{{D}}
\renewcommand{\S}{S}
\newcommand{\aq}{\overrightarrow{q}}
\newcommand{\ap}{\overrightarrow{p}}
\renewcommand*{\dot}[1]{\accentset{\mbox{\large\bfseries .}}{#1}}
\newcommand{\Leb}{\mathsf{Leb}}
\newcommand{\comple}{{\complement}}
\newcommand{\T}{T}
\newcommand{\one}{\mathbf{1}}
\newcommand{\brho}{\mathbf{h}}
\newcommand{\cM}{\mcl M}
\newcommand{\bi}{\mathbf{i}}

\newcommand{\sS}{\mathscr{S}}
\newcommand{\fR}{{\mathfrak{R}}}

\newcommand{\s}{{\mathsf{s}}}

\begin{document}

\author{Hong-Bin Chen}
\address[Hong-Bin Chen]{Institut des Hautes \'Etudes Scientifiques, France}
\email{hbchen@ihes.fr}

\author{Jean-Christophe Mourrat}
\address[Jean-Christophe Mourrat]{Department of Mathematics, ENS Lyon and CNRS, Lyon, France}
\email{jean-christophe.mourrat@ens-lyon.fr}

\subjclass[2010]{}
\date{\today}

\title[Simultaneous RSB for vector spin glasses]{Simultaneous replica-symmetry breaking \\
for vector spin glasses}

\begin{abstract}
We consider mean-field vector spin glasses with possibly non-convex interactions. Up to a small perturbation of the parameters defining the model, the asymptotic behavior of the Gibbs measure is described in terms of a critical point of an explicit functional. In this paper, we study some properties of these critical points. Under modest assumptions ensuring that different types of spins interact, we show that the replica-symmetry-breaking structures of the different types of spins are in one-to-one correspondence with one another. For instance, if some type of spins displays one level of replica-symmetry breaking, then so do all the other types of spins. This extends the recent results of \cite{bates2022free, bates2022crisanti} that were obtained in the case of multi-species spherical spin glasses with convex interactions.
\end{abstract}

\maketitle

\section{Introduction}

We consider general multi-species or vector spin glasses with possibly non-convex interactions. When the interaction is indeed non-convex, a complete identification of the limit free energy is still lacking. Despite this, it was shown in \cite{chen2023free} that the limit free energy and overlap distribution must satisfy strong constraints, as they must admit a representation in terms of a critical point of an explicit functional. 
The goal of the present paper is to demonstrate that this representation allows us to infer physically relevant information about the limit behavior of the model. We do so in the context of the property of simultaneous replica-symmetry breaking mentioned as a question in \cite{pan.multi} and investigated in \cite{bates2022free, bates2022crisanti} for multi-species spherical spin glasses with convex interactions. Roughly speaking, it was shown in \cite{bates2022free, bates2022crisanti} that for such models and under mild conditions ensuring the effective coupling between the species, it must be that all species of spins share the same replica-symmetry-breaking structure. For instance, it cannot be that one species is replica-symmetric (its overlap is concentrated) while another one has one level of symmetry breaking (its overlap is asymptotically supported on two values). We show here that this property extends to non-spherical vector spin glasses with possibly non-convex interactions.

We fix an integer $D \ge 1$, and let $(H_N(\sigma))_{\sigma \in (\R^N)^D}$ be a centered Gaussian field such that, for every $\sigma = (\sigma_1,\ldots, \sigma_D)$ and $\tau = (\tau_1, \ldots, \tau_D) \in (\R^N)^D$, we have
\begin{equation}
\label{e.def.xi}
\E \Ll[ H_N(\sigma) H_N(\tau) \Rr] = N \xi \Ll( \frac{\si \tau^\intercal}{N} \Rr) ,
\end{equation}
where $\xi \in C^\infty(\R^{D\times D}; \R)$ is a smooth function admitting an absolutely convergent power-series expansion with $\xi(0) = 0$, and where $\sigma \tau^\intercal$ denotes the matrix of scalar products 
\begin{equation}
\label{e.def.sigma.tau}
\sigma \tau^\intercal = (\sigma_d \cdot \tau_{d'})_{1 \le d,d' \le D}.
\end{equation}
The notation in \eqref{e.def.sigma.tau} is natural if we think of $\si$ and $\tau$ as $D$-by-$N$ matrices, and we often identify $(\R^N)^D$ with $\R^{D \times N}$. We think of the Gaussian fields that can be represented in the form of \eqref{e.def.xi} as essentially encoding the most general class of (non-sparse) mean-field spin glasses; the mean-field character of the model is captured by the assumption that the covariance in \eqref{e.def.xi} depends only on the matrix of scalar products between the different types of spins. The set of functions $\xi$ such that there exists a Gaussian random field satisfying~\eqref{e.def.xi}, together with an explicit construction of the Gaussian process $H_N$, can be found in \cite[Proposition~6.6]{mourrat2023free} and \cite[Subsection~1.5]{chen2023free}. We do \emph{not} impose any convexity assumption on $\xi$. 

We give ourselves a probability measure $P_1$ with compact support in $\R^D$, and for each integer $N \ge 1$, we denote by $P_N = P_1^{\otimes N}$ the $N$-fold tensor product of $P_1$. We think of $P_N$ as a probability measure on $\R^{D \times N} \simeq (\R^N)^D$. We will often also assume that
\begin{equation}
\label{e.ass.full.support}
\mbox{the affine space spanned by the support of $P_1$ is the full space $\R^D$.}
\end{equation}
This entails no loss of generality, for whenever it is not satisfied, we can always redefine the model by restricting our attention to the affine space spanned by the support of $P_1$.

For each $\beta \ge 0$, we are interested in the large-$N$ behavior of the Gibbs measure whose Radon--Nikodym derivative with respect to $P_N$ is proportional to $\exp(\beta H_N)$, and of the corresponding free energy. 
For technical reasons, we will in fact consider a more complicated ``enriched'' version of the free energy. We denote by $S^D_+$ the cone of positive semidefinite matrices.
We let $\mcl Q$ be the space of increasing and right-continuous paths $q : [0,1) \to S^D_+$ with left limits; a path $q : [0,1) \to S^D_+$ is said to be increasing if for every $u \le v \in [0,1)$, we have $q(u) \le q(v)$ in the sense that $q(v) - q(u) \in S^D_+$. For every $r \in [1,\infty]$, we write $\mcl Q_r = \mcl Q \cap L^r([0,1], S^D)$; we also often use the shorthand $L^r$ for $L^r([0,1], S^D)$. The enriched free energy $\bar F_N(t,q)$ is defined for every $t \in \R_+$ and $q \in \mcl Q_1$. When the path~$q$ is constantly equal to $h \in S^D_+$, this free energy is given by 
\begin{multline}  
\label{e.def.simple.F}
\bar F_N(t,h) 
\\
= -\frac 1 N \E \log \int \exp \Ll( \sqrt{2t} H_N(\sigma) - t N \xi\Ll(\frac{\sigma\sigma^\intercal}{N}\Rr) + \sqrt{2h} z \cdot \sigma - h \cdot \sigma \sigma^\intercal \Rr) \, \d P_N(\sigma),
\end{multline}
where $z = (z_1,\ldots, z_D)$ is a random element of $(\R^{N})^D \simeq \R^{D\times N}$ with independent standard Gaussian entries, independent of $(H_N(\sigma))_{\sigma \in (\R^N)^D}$, and $\E$ denotes the expectation with respect to all randomness. 
In~\eqref{e.def.simple.F} and throughout the paper, whenever $a$ and $b$ are two matrices of the same size, we denote by $a \cdot b = \tr(a b^\intercal)$ the entrywise scalar product, and by $|a| = (a \cdot a)^{1/2}$ the associated norm. 
The additional term $t N \xi\Ll(\frac{\sigma\sigma^\intercal}{N}\Rr)$ in \eqref{e.def.simple.F} is introduced as a technical convenience; when $\xi$ is convex, a procedure for removing it a posteriori is explained in \cite{chen2024free, mourrat2020extending}; otherwise, more work is necessary, but at least this term is actually constant for some non-convex models of interest, such as the bipartite model considered e.g.\ in \cite{mourrat2021nonconvex} with $\pm 1$ spins.
For non-constant paths $q$, the definition of $\bar F_N(t,q)$ involves replacing the random magnetic field $\sqrt{2h} z$ by a more complicated external field  comprising an ultrametric structure, and the parameters of this structure are encoded into the path~$q$. We denote by $\la \cdot \ra$ the Gibbs measure associated with the free energy $\bar F_N(t,q)$; the precise definitions of the free energy and the Gibbs measure are given in Subsection~\ref{ss.def.free.energy} below. We stress that the Gibbs measure~$\la \cdot \ra$ is still random as it depends on the realization of the Gaussian random field~$H_N$ and of the external magnetic field. Although this is not displayed in the notation, it also depends on the choice of the parameters $N$, $t$ and $q$. We denote by $\sigma$ the canonical spin variable under $\la \cdot \ra$, and by $\sigma'$ an independent copy of $\sigma$ under $\la \cdot \ra$, also called a replica. In the limit of large $N$, we can probe the geometry of the Gibbs measure by studying the law of the overlap $\sigma \sigma'^\intercal/N$ under $\E \la \cdot \ra$ (see for instance \cite[Section 5.7]{HJbook} for a more thorough discussion of this point).  One of the main results of \cite{chen2023free} states that up to a small perturbation of the energy function, any subsequential limit of the law of the overlap must be encoded by a critical point of an explicit functional which we describe next. 

To motivate the definition of this functional, we start by pointing out that, whenever the function $\xi$ is convex over $S^D_+$, we have (see \cite[Theorem~1.1]{chen2023free}) that $\bar F_N$ converges pointwise to the viscosity solution $f$ to the equation
\begin{equation}
\label{e.hj}
\partial_t f - \int_0^1 \xi(\partial_q f) = 0,
\end{equation}
with initial condition
\begin{equation}
\label{e.def.psi}
\psi(q) = \lim_{N \to + \infty} \bar F_N(0,q). 
\end{equation}
Moreover, whether or not $\xi$ is convex,  every subsequential limit $f$ of $\bar F_N$ satisfies the relation \eqref{e.hj} at every point of differentiability of $f$, by \cite[Proposition~7.2]{chen2023free} (see also  \cite{chen2022hamilton, mourrat2021nonconvex, mourrat2023free} for related works exploring the links between the limit free energy and the partial differential equation in \eqref{e.hj}, and \cite{HJbook} for a book presentation). 
The derivative $\partial_q f$ in \eqref{e.hj} is understood in the sense that for every $q' \in \mcl Q_2$ and as $\ep > 0$ tends to~$0$, we have
\begin{equation*}  f(t,q + \ep(q'-q)) - f(t,q) = \ep \int_0^1 (q'-q)(u) \cdot \dr_q f(t,q,u) \, \d u + o(\ep),
\end{equation*}
and the integral in \eqref{e.hj} can be written more explicitly as $\int_0^1 \xi(\dr_q f(t,q,u)) \, \d u$. 
Under our assumption that $P_N=P_1^{\otimes N}$, the quantity $\bar F_N(0,q)$ in fact does not depend on $N$ (see \cite[Proposition~3.2]{mourrat2023free}), so we have $\psi(q) = \bar F_1(0,q)$ in this case. Although we will not study spherical models here, we point out that when $P_N$ is the uniform measure on the sphere of radius $\sqrt{N}$ in~$\R^N$, an explicit expression for $\psi$ is given in~\cite[Proposition~3.1]{mourrat2022parisi}.

For every $t \ge 0$, $q,q' \in \mcl Q_2$, and $p \in L^2 (= L^2([0,1], S^D))$, we set
\begin{align*}
\mcl J_{t,q}(q',p) & = \psi(q') + \int_0^1 p(u) \cdot (q(u) - q'(u)) \, \d u + t \int_0^1 \xi(p(u)) \, \d u 
\\
\notag & = \psi(q') + \int_0^1 p \cdot(q-q') + t \int_0^1 \xi(p).
\end{align*}
The functional $\mcl J_{t,q}$ is closely related to the Hamilton--Jacobi equation in \eqref{e.hj}, so we call it the \emph{Hamilton--Jacobi functional}. A first simple link between the two objects is the observation that, for each fixed $q'$ and $p$, the mapping $(t,q) \mapsto \mcl J_{t,q}(q',p)$ is a solution to~\eqref{e.hj}. More profound links are presented in detail in \cite[Sections~3.4 and 3.5]{HJbook}. We say that a pair $(q',p) \in \mcl Q_2 \times L^2$ is a \emph{critical point} of $\mcl J_{t,q}$ if
\begin{equation}
\label{e.crit.point}
q = q' - t \nabla \xi(p) \quad \text{ and } \quad p = \partial_q \psi(q').
\end{equation}
The first condition is the critical-point condition with respect to variations of the parameter~$p$, while the second one relates to variations with respect to the parameter $q'$. In a nutshell, the main result of \cite{chen2023free} is that possibly up to a small perturbation of the parameters of the model and up to the extraction of a subsequence, we have that the overlap matrix $\sigma \sigma'^\intercal/N$ converges in law to $p(U)$, where $(q',p) \in \mcl Q_\infty$ is a critical point of~$\mcl J_{t,q}$ and $U$ is a uniform random variable over $[0,1]$. 

We therefore seek to study the properties of critical points of $\mcl J_{t,q}$. Prior works related to this goal include \cite{auffinger2015properties, auffinger2020sk, auffinger2019existence, bates2019replica, dey2021fluctuation, jagannath2017low, jagannath2018bounds, montanari2003nature, talagrand2000multiple, talagrand2006free, talagrand2006parisi, Tbook1, Tbook2}. Much of the attention there is paid to the number of steps of replica-symmetry breaking. In our context, the number of steps of replica-symmetry breaking, plus one, is the number of different values taken by the path~$p$. 

For multi-species spherical models with convex interactions, the property of simultaneous replica-symmetry breaking has been investigated in \cite{bates2022free, bates2022crisanti} (see also \cite{ko2019crisanti} concerning the Crisanti-Sommers formula for the limit free energy). This refers to the idea that, provided that the different species of spins are effectively coupled in the model, the supports of the limit laws of the overlaps of the different species are in bijective correspondence with one another. In other words, all species share the same number of levels of replica-symmetry breaking. In the present paper, we extend this result to non-spherical models with possibly non-convex interactions. 

For clarity, we appeal to a simple assumption on $\xi$ to guarantee the effective coupling between the different types of spins; more general results will be discussed later on. As explained in \cite[Section~6]{mourrat2023free}, for $\xi$ to be the covariance of a Gaussian field as in \eqref{e.def.xi}, it is necessary that for every $a,b \in S^D_+$, one has $\nabla \xi(a + b) - \nabla \xi(a) \in S^D_+$. A sufficient assumption to ensure that the different types of spins are indeed coupled is that 
\begin{equation}
\label{e.ass.coupling}
\mbox{for every $a,b \in S^D_+$ with $b \neq 0$, we have $    \frac{\d}{\d \eps}\Big|_{\eps=0}\nabla\xi(a+\eps b)\in S^D_{++}$,}
\end{equation}
where $S^D_{++}$ denotes the set of positive definite matrices.
As explained in more details in Remark~\ref{r.example}, one possible way to ensure the validity of this assumption is to add to the energy function $H_N$ a term of the form
\begin{equation}  
\label{e.extra.energy}
\sum_{d,d' = 1}^D \sqrt{c_{d,d'}} \sum_{i,j = 1}^N W_{i,j}^{d,d'} \sigma_{d,i} \sigma_{d',j},
\end{equation}
where $(c_{d,d'})_{d,d' \le D}$ are in $(0,+\infty)$ and $(W_{i,j}^{d,d'})_{i,j \le N, d,d' \le D}$ are independent standard Gaussian random variables, independent of $H_N$. 
\begin{theorem}[Simultaneous RSB]
\label{t.main}
Suppose that the assumptions \eqref{e.ass.full.support} and \eqref{e.ass.coupling} hold. Let $t > 0$, $q \in \mcl Q_1$, and let $(q',p) \in \mcl Q_{\infty}^2$ be a critical point of $\mcl J_{t,q}$. For every $s < s' \in [0,1]$, if $p(s') - p(s) \neq 0$, then $p(s') - p(s) \in S^D_{++}$. 
\end{theorem}

Recall that $p \in \mathcal{Q}_\infty$ is originally defined on $[0,1)$. In Theorem~\ref{t.main}, $p(1)$ is understood as $\lim_{u \to 1} p(u)$, as explained around~\eqref{e.q(1)=}.

Let us first explain why Theorem~\ref{t.main} relates to the statement of simultaneous replica-symmetry-breaking. Recall that by \cite[Theorem~1.4]{chen2023free}, possibly up to a small perturbation of the parameters and up to the extraction of a subsequence, the law of the overlap $\sigma \sigma'^\intercal/N$ converges to $p(U)$, where $(q',p) \in \mcl Q_\infty$ is a critical point of $\mcl J_{t,q}$ and $U$ is a uniform random variable over $[0,1]$. Recalling that $p$ is a matrix-valued path, we denote by $p_1,\ldots, p_D$ its diagonal entries, so that the limit law of the overlap of the $d$-th spin type is $p_d(U)$. Suppose that the $d$-th spin type has at least $K$ levels of replica-symmetry breaking. This means that $p_d$ takes at least $K+1$ values, so we can find $0 \le s_0 < s'_0 < s_1 < s'_1 < \cdots s_K < s'_K\le 1$ such that for each $k \in \{0,\ldots, K\}$, we have that $p_d(s_k) < p_d(s_k')$. An application of Theorem~\ref{t.main} then yields that all the other spin types also have at least $K$ levels of replica-symmetry breaking. Since this argument applies for every choice of $d \in \{1, \ldots, D\}$, we conclude that the paths $(p_d)_{d \le D}$ all take exactly the same number of values, as desired. 

The assumption \eqref{e.ass.coupling} is not necessary to ensure the validity of the phenomenon of simultaneous replica-symmetry breaking, and more general results will be proved below. For instance, it suffices that the energy function contain a term of the form of \eqref{e.extra.energy} with $c_{d,d'}$ non-zero only for $d' = d + 1 \ (\mathrm{mod}\ D)$. What is key is that there is a ``transmission chain'' from each spin type to every other spin type; see also the notion of ``$y$-to-$z$ coupled'' in Proposition~\ref{p.sim_RSB} and Remark~\ref{r.example}. 

We also prove an analogue of Theorem~\ref{t.main} for multi-species models, which is perhaps simpler to interpret as in this context we can ignore the off-diagonal part of the paths. In other words, the statement in this case is directly at the level of the paths $(p_1,\ldots, p_D)$ that encode the respective laws of the overlaps for each of the different species; see Theorem~\ref{t.sim_RSB_MS} for more precision.

As a side remark, we mention that the convergence in law of the overlap provided by \cite[Theorem~1.4]{chen2023free} is only known to be valid with the caveats stated above (that is, up to an arbitrarily small perturbation of the parameters and up to the extraction of a subsequence). These caveats cannot be removed altogether, as for instance, for the SK model ($D = 1$, $\xi(r) = r^2$, $P_1 = \frac 1 2 (\de_1 + \de_{-1})$), symmetry considerations yield that the law of the overlap is symmetric under the transformation $z \mapsto -z$, which would be in contradiction with the stated convergence in law at low temperature, where the limit law is non-trivial and supported in $[0,1]$. Yet, when $\xi$ is convex and the path $q$ is chosen to be strictly increasing, the convergence is shown to hold with no caveats in \cite[Corollary~8.7]{chen2023free}. This statement is analogous to the one obtained for the scalar models called generic in \cite[Section~3.7]{pan}.

\medskip 

\noindent \textbf{Organization of the paper.} In Section~\ref{s.cascade}, we recall the construction of the free energy $\bar F_N(t,q)$ based on Poisson--Dirichlet cascades and present some basic properties of these objects. For our purposes, we need to manipulate possibly discontinuous matrix-valued paths. In Section~\ref{s.decomp}, we introduce convenient decompositions of such a path as the composition of a Lipschitz matrix-valued path and a quantile function. Section~\ref{s.parisi.pde} presents fundamental properties of the Parisi PDE, which appears in the explicit calculation of the initial condition $\psi$ in \eqref{e.def.psi}. In Section~\ref{s.proof}, we obtain convenient representations of the derivative of $\psi$, which may be of independent interest, and we use them to prove Theorem~\ref{t.main}. The final section covers the case of multi-species models.

\section{Continuous cascades}\label{s.cascade}

In this section, we recall definitions and properties of continuous Poisson--Dirichlet cascades and associated Gaussian processes. We also recall the definition of the free energy with an external field parameterized by a matrix-valued path, which generalizes~\eqref{e.def.simple.F}.

\subsection{Definition and properties of the continuous
cascade}The existence of the following objects has been explained in the beginning of~\cite[Section~4]{chen2023free}.
There is an infinite-dimensional separable Hilbert space $\mathfrak{H}$ and a random probability measure $\mathfrak{R}$ on the unit sphere of $\mathfrak{H}$ such that the following holds. We denote the inner product in $\mathfrak{H}$ by $\wedge$ and let $\Ll(\brho^l\Rr)_{l\in\N}$ be an sequence of independent random variables sampled from $\mathfrak{R}$. Then, under $\P\mathfrak{R}^\N$, $\brho^l\wedge \brho^{l'}$ is distributed uniformly on $[0,1]$ for every distinct pair $l$ and $l'$. Moreover, for almost every realization of $\mathfrak{R}$, the support of $\mathfrak{R}$ satisfies the following properties: (i) $\supp\mathfrak{R}$ is a subset of the unit sphere; (ii) $\brho\wedge\brho'\geq 0$ for all $\brho,\brho'\in\supp\mathfrak{R}$; (iii) $\brho\wedge\brho'\geq \min\Ll(\brho\wedge\brho'',\brho''\wedge\brho'\Rr)$, for all $\brho,\brho',\brho''\in \supp\mathfrak{R}$.

By~\cite[Proposition~4.1]{chen2023free}, we know the existence of the following Gaussian process. For almost every realization of $\mathfrak{R}$ and every $q\in\mcl Q_\infty$, there is an $\R^\D$-valued centered Gaussian process $\Ll(w^q(\brho)\Rr)_{\brho\in\supp\mathfrak{R}}$ such that for every $\brho,\brho'\in\supp\mathfrak{R}$,
\begin{align}\label{e.E[ww]=}
    \E \Ll[w^q(\brho)w^q(\brho')^\intercal\Rr] = q\Ll(\brho\wedge\brho'\Rr).
\end{align}
Here, we recall that $q\in \mcl Q_\infty$ is a function defined on $[0,1)$. Since $q$ is bounded and increasing, we have that the function $s\mapsto a\cdot q(s)$ for every $a\in\S^\D_+$ is bounded and increasing. Hence, $\lim_{s\to1} a\cdot q(s)$ exists for every $a\in\S^\D_+$, then we can define $q(1)\in\S^\D_+$ to be the unique matrix determined by $\lim_{s\to1} a\cdot q(s) = a\cdot q(1)$. In this way, we can extend the definition of $q$ to $[0,1]$ with the right endpoint satisfying
\begin{align}\label{e.q(1)=}
    q(1) = \lim_{s\to1}q(s).
\end{align}
Then,~\eqref{e.E[ww]=} makes sense when $\brho=\brho'$ or equivalently $\brho\wedge\brho'=1$.

Henceforth, whenever we take expectations with respect to $\mathfrak{R}$ and $\Ll(w^q(\brho)\Rr)_{\brho\in\supp\mathfrak{R}}$, we always first average the Gaussian randomness of $\Ll(w^q(\brho)\Rr)_{\brho\in\supp\mathfrak{R}}$ conditioned on $\mathfrak{R}$ and then average over the randomness of $\mathfrak{R}$. According to~\cite[Lemma~4.5]{chen2023free}, this order of averaging is needed to avoid measurability issues in all situations that we are interested in.

We often appeal to the following invariance property of the cascade. For a Lipschitz $\mathbf g:\R^\D\to \R$, we consider the random Gibbs measure
\begin{align*}
    \la \cdot \ra_{\mathfrak{R}^\mathbf{g}} = \frac{\exp \Ll(g(w^q(\brho)\Rr)\d \mathfrak{R}(\brho)}{\iint \exp \Ll(g(w^q(\brho)\Rr)\d \mathfrak{R}(\brho)}
\end{align*}
which averages over $\brho$. In the following, we often use $(\sigma,\brho)$ and $(\sigma',\brho')$ to denote two independent copies sampled from a random Gibbs measure.

\begin{lemma}[Invariance of cascades]\label{l.invar_cascade}
For every $q\in \mcl Q_\infty$, every Lipschitz $\mathbf{g}:\R^\D\to\R$, and every bounded measurable $\rho:\R\to \R$, we have
\begin{align*}
    \E \la \rho \Ll(\brho\wedge\brho'\Rr) \ra_{\mathfrak{R}^\mathbf{g}} = \int_0^1 \rho(s) \d s.
\end{align*}
\end{lemma}

\begin{proof}
According to~\cite[Proposition~4.8]{chen2023free}, the left-hand side is unchanged if we replace $\mathbf{g}$ by the constantly zero function. Then, the identity follows from the fact that $\brho\wedge\brho'$ is uniformly distributed over $[0,1]$ under $\E \la \cdot \ra_{\mathfrak{R}^0}$.
\end{proof}

We need the following standard interpolation computation.

\begin{lemma}\label{l.interpolation}
Let $\mu$ be a finite nonzero Borel measure supported on the unit ball of $\R^\D$. For every $q\in \mcl Q_\infty$, $x\in\R^\D$, and $z\in\R^{D\times D}$, we set
\begin{gather}\label{e.f(q),<>_q=}
\begin{split}
    \mathbf{f}_\mu(q,x,z) &= \E \log \iint \exp \Ll(\sqrt{2}\sigma\cdot w^q(\brho)+\sigma\cdot x + \sigma\sigma^\intercal\cdot(z-q(1))\Rr)\d \mu(\sigma) \d \mathfrak{R}(\brho),
    \\
    \la \cdot \ra_{\mu,q,x,z} &= \frac{\exp \Ll(\sqrt{2}\sigma\cdot w^q(\brho)+\sigma\cdot x+ \sigma\sigma^\intercal\cdot(z-q(1))\Rr)\d \mu(\sigma) \d \mathfrak{R}(\brho)}{\iint \exp \Ll(\sqrt{2}\sigma\cdot w^q(\brho)+\sigma\cdot x+ \sigma\sigma^\intercal\cdot(z-q(1))\Rr)\d \mu(\sigma) \d \mathfrak{R}(\brho)}.
\end{split}
\end{gather}
Then, for every $q,q'\in\mcl Q_\infty$, every $x,x'\in \R^\D$, and every $z,z'\in\R^{\D\times\D}$, writing $q_\lambda = \lambda q+(1-\lambda)q'$, $x_\lambda =\lambda x+(1-\lambda)x'$, and $z_\lambda =\lambda z+(1-\lambda)z'$, we have
\begin{align}\label{e.f(q)-f(q')=int}
    \mathbf{f}_\mu(q,x,z) - \mathbf{f}_\mu(q',x',z')= \int_0^1\E \la R_{1,1}-R_{1,2}  + \sigma^1\cdot (x-x')\ra_{\mu,q_\lambda,x_\lambda,z_\lambda} \d \lambda;
\end{align}
and, for any bounded measurable function $\mathbf{F}:\Ll(\R^\D\Rr)^n \times \R^{n\times n}\to \R$ with $n\in\N$, there are constants $(c_{l,l'})_{1\leq l,l'\leq n+1}$ and $(c_l)_{1\leq l\leq n+1}$ depending only on $n$ such that
\begin{align}\label{e.E<F>_q-E<F>_q'}
\begin{split}
     &\E\la \mathbf{F}(\cdots) \ra_{\mu,q,x,z} - \E\la \mathbf{F}(\cdots)\ra_{\mu,q',x',z'}
    \\
    &= \int_0^1\E \la \mathbf{F}(\cdots) \Ll(\sum_{l,l'=1}^{n+1} c_{l,l'}R_{l,l'} + \sum_{l=1}^{n+1}c_l \sigma^l\cdot (x-x')\Rr)\ra_{\mu,q_\lambda,x_\lambda,z_\lambda} \d \lambda
\end{split}
\end{align}
where $\mathbf{F}(\cdots) = \mathbf{F}\Ll(\sigma^1,\dots ,\sigma^n;\; \Ll(\brho^l\wedge\brho^{l'}\Rr)_{1\leq l,l'\leq n}\Rr)$ and
\begin{align*}
    R_{l,l'}= 
    \begin{cases}
        \sigma^l\Ll(\sigma^{l}\Rr)^\intercal \cdot \Ll(z-z'\Rr),\quad &\text{if $l= l'$},
        \\\sigma^l\Ll(\sigma^{l'}\Rr)^\intercal \cdot \Ll(q-q'\Rr)\Ll(\brho^l\wedge\brho^{l'}\Rr),\quad &\text{if $l\neq l'$}.
    \end{cases}
\end{align*}
\end{lemma}

\begin{proof}
Taking $\lambda\in[0,1]$ and $w^q$ to be independent from $w^{q'}$, we can see that $\sqrt{\lambda} w^q + \sqrt{1-\lambda}w^{q'}$ has the same distribution as $w^{\lambda q + (1-\lambda)q'}$ by checking that they have the same covariance using~\eqref{e.E[ww]=}. Hence, we can rewrite
\begin{align*}
    \mathbf{f}_\mu\Ll(q_\lambda,x_\lambda,z_\lambda\Rr) 
    = \E \log \iint \exp \Big(\sqrt{2}\sigma \cdot \Ll(\sqrt{\lambda}w^{q}(\brho) + \sqrt{1-\lambda }w^{q'}(\brho)\Rr)  
    \\
    + \sigma\cdot x_\lambda + \sigma\sigma^\intercal\cdot(z_\lambda-q_\lambda(1))\Big) \d \mu(\sigma) \d\mathfrak{R}(\brho).
\end{align*}
Then, for $\lambda\in(0,1)$, we first compute the derivative and then use the Gaussian integration by parts (c.f.\ \cite[Theorem~4.6]{HJbook}) to get
\begin{align*}
	&\frac{\d}{\d\lambda} \mathbf{f}_\mu\Ll(q_\lambda,x_\lambda,z_\lambda\Rr)
    \\
	& = \E \la \sigma\cdot \Ll(\frac{1}{\sqrt{2\lambda}}w^{q}(\brho)-\frac{1}{\sqrt{2(1-\lambda)}}w^{q'}(\brho)+x-x'\Rr)  + \sigma\sigma^\intercal\cdot (z-z'-(q-q')(1))\ra_{\mu,q_\lambda,x_\lambda}
	\\
	& = \E \la \sigma\sigma^\intercal \cdot (z-z')   -   \sigma\sigma'^\intercal \cdot \Ll(q-q'\Rr)\Ll(\brho\wedge\brho'\Rr)  +\sigma\cdot(x-x')\ra_{\mu,q_\lambda,x_\lambda}
\end{align*}
which yields~\eqref{e.f(q)-f(q')=int}. For~\eqref{e.E<F>_q-E<F>_q'}, we can use the same interpolation arguments based on replacing $w^{\lambda q + (1-\lambda)q'}$ with $\sqrt{\lambda} w^q + \sqrt{1-\lambda}w^{q'}$.
This time, the expression is more complicated due to the presence of $n$ independent copies of $(\sigma,\brho)$ from the differentiation and $n+1$ copies after performing the Gaussian integration by parts. 
Since the procedure is standard, we omit the details here.
\end{proof}

\begin{corollary}\label{c.interpolation}
Under the same setting as in Lemma~\ref{l.interpolation}, we have
\begin{align*}
    \Ll|\mathbf{f}_\mu(q,x,z) - \mathbf{f}_\mu(q',x',z')\Rr|\leq \Ll|z -z'\Rr| + \Ll\|q-q'\Rr\|_{L^1} + |x-x'|,
\end{align*}
and, there is a constant $C_n$ depending only on $n$ such that
\begin{align*}
    \Ll|\la \mathbf{F}(\cdots) \ra_{\mu,q,x,z} - \la \mathbf{F}(\cdots)\ra_{\mu,q',x',z'}\Rr|\leq C_n \|\mathbf{F}\|_{L^\infty} \Ll(\Ll|z -z'\Rr| + \Ll\|q-q'\Rr\|_{L^1}+|x-x'|\Rr).
\end{align*}
\end{corollary}

\begin{proof}
Due to the assumption on the support of $\mu$, we have $\Ll|R_{l,l'}\Rr|\leq \Ll|q-q'\Rr|\Ll(\brho^l\wedge\brho^{l'}\Rr)$ for $l\neq l'$ and $\Ll|R_{l,l}\Rr|\leq \Ll|z-z'\Rr|$. Using this and Lemma~\ref{l.invar_cascade} for $\brho^l\wedge\brho^{l'}$ with $l\neq l'$, we obtain the desired results from~\eqref{e.f(q)-f(q')=int} and~\eqref{e.E<F>_q-E<F>_q'}.
\end{proof}

\subsection{Enriched free energy}
\label{ss.def.free.energy}

Lastly, we define the enriched free energy. For almost every realization of $\mathfrak{R}$, let $(w^q_i)_{i\in\N}$ be independent copies of $w^q$ and define, for every $N\in\N$ and $\brho \in \supp\mathfrak{R}$,
\begin{align*}
    W^q_N(\brho) = \Ll(w^q_1(\brho),\dots ,w^q_N(\brho)\Rr)
\end{align*}
which takes value in $\R^{\D\times N}$ (namely, column vectors of $W^q_N(\brho)$ are given by $w^q_i(\brho)$).
For every $N\in\N$ and $(t,q)\in \R_+\times \mcl Q_\infty$,
we consider the Hamiltonian
\begin{align*}
    H^{t,q}_N(\sigma,\brho) = \sqrt{2t}H_N(\sigma)- t N\xi\Ll(\frac{\sigma\sigma^\intercal}{N}\Rr)+\sqrt{2} W^q_N(\brho)\cdot \sigma - q(1)\cdot \sigma\sigma^\intercal
\end{align*}
and the free energy
\begin{align*}\bar F_N(t,q) = - \frac{1}{N}\E \log \iint \exp\Ll(H^{t,q}_N(\sigma,\brho)\Rr)\d P_N(\sigma)\d \mathfrak{R}(\brho),
\end{align*}
where $\E$ first averages over the Gaussian randomness in $H^{t,q}_N(\sigma,\brho)$ and then the randomness of $\mathfrak{R}$.
The associated random Gibbs measure (averaging over $(\sigma,\brho)$) is defined by
\begin{align}\label{e.<>_N=}
    \la \cdot \ra_N = \frac{\exp\Ll(H^{t,q}_N(\sigma,\brho)\Rr)\d P_N(\sigma)\d \mathfrak{R}(\brho)}{\iint \exp\Ll(H^{t,q}_N(\sigma,\brho)\Rr)\d P_N(\sigma)\d \mathfrak{R}(\brho)}.
\end{align}

Viewing $\bar F_N$ as a function on $\R_+\times \mcl Q_\infty$, we can interpret $\bar F_N(0,\cdot)$ as its initial condition. Due to the assumption $P_N=P_1^{\otimes N}$, we have $\bar F_N(0,\cdot) = \bar F_1(\cdot)$ (\cite[Proposition~3.2]{chen2023free}). Then, $\psi$ given as in~\eqref{e.def.psi} is equal to $\bar F_1(0,\cdot)$ and has the explicit expression:
\begin{align}\label{e.psi=}
    \psi (q) =  - \E \log \iint \exp\Ll(\sqrt{2}\sigma\cdot w^{q}(\brho) - q(1)\cdot \sigma\sigma^\intercal\Rr) \d P_1(\sigma) \d \mathfrak{R}(\brho),\quad\forall q\in \mcl Q_\infty.
\end{align}
It is known (\cite[Definition~2.1 and Corollary~5.2]{chen2023free}) that $\psi$ is Fr\'echet differentiable at every $q\in\mcl Q_\infty$ in the following sense. For every $q\in \mcl Q_\infty$, there is a unique $p\in \mcl Q_\infty$ such that
\begin{align}\label{e.psi_diff}
    \lim_{r\to0}\sup_{\substack{q' \in \mcl Q_\infty \setminus\{0\}\\\|q'-q\|_{L^2} \leq r}} \frac{\Ll|\psi(q')-\psi(q)-\la p, q'-q\ra_{L^2}\Rr|}{\|q'-q\|_{L^2}} =0,
\end{align}
where $\la\cdot,\cdot\ra_{L^2}$ is the $L^2$ scalar product for $\S^\D$-valued functions defined on $[0,1]$. We denote this unique $p$ by $\partial_q \psi(q) \in \mcl Q_\infty$.

\begin{remark}\label{r.left-cts}
In this section, objects are defined with $q\in \mcl Q_\infty$ which is a right-continuous path with left limits. Later, we will introduce the left-continuous version $\aq$ of $q$ defined in~\eqref{e.aq=}, which is more natural for other purposes. Here, we comment that all the objects can be defined in terms of $\aq$ instead of $q$ because conditioned on $\mathfrak{R}$ we can construct the Gaussian process $\Ll(w^{\aq}(\brho)\Rr)_{\brho\in\supp\mathfrak{R}}$ with covariance
\begin{align*}
    \E \Ll[w^{\aq}(\brho)w^{\aq}(\brho')^\intercal\Rr] = {\aq}\Ll(\brho\wedge\brho'\Rr),\quad \forall \brho,\brho'\in\supp\mathfrak{R}.
\end{align*} 
Relevant properties are also preserved, as explained in~\cite[Remark~4.9]{chen2023free}. In particular, the invariance property in Lemma~\ref{l.invar_cascade} still holds for this version. Since $q$ and $\aq$ differ on a set with zero Lebesgue measure, using interpolation arguments as in Lemma~\ref{l.interpolation} and the invariance property, we can see that $\mathbf{f}_\mu(q,x,z)$ and the deterministic measure $\E \la\cdot\ra_{q,x,z}$ are preserved if we change $w^q$ in~\eqref{e.f(q),<>_q=} to $w^{\aq}$.
\qed
\end{remark}

\section{Decomposition of matrix-valued paths}
\label{s.decomp}

A path $q\in\mcl Q_\infty$ is matrix-valued. To facilitate the definition of the Parisi PDE along the path $q$, we need to perform a decomposition of $q$ into a Lipschitz matrix-valued path and a (scalar-valued, possibly discontinuous) quantile function. We explain this procedure in this section and derive necessary properties for later use.

\subsection{Quantile functions}

On an interval $[0,\T]$ for some $\T>0$, the function $\alpha:[0,\T]\to [0,1]$ is said to be a \textbf{probability distribution function} (p.d.f.)\ if $\alpha$ is increasing (by this we mean that $\alpha(t)\leq \alpha(t')$ whenever $t \leq t'$), is right-continuous with left limits, and satisfies $\alpha(\T)=1$. Although not needed, we can extend $\alpha$ to $\bar\alpha$ defined on the entire real line by setting $\bar\alpha(t)=0$ for $t<0$ and $\bar\alpha(t)=1$ for $t>\T$. We denote the associated probability measure by $\d\alpha$; this probability measure is uniquely determined by the property that $\int_{(-\infty,t]}\d \alpha = \bar\alpha(t)$ for every $t\in\R$.

The \textbf{quantile function} $\alpha^{-1}:[0,1]\to[0,\T]$ associated with $\alpha$ is given by
\begin{align}\label{e.alpha^-1=}
    \alpha^{-1}(s) = \inf\Ll\{t\in[0,\T]:\: s\leq \alpha(t)\Rr\},\quad\forall s\in [0,1].
\end{align}
It is straightforward to see that $\alpha^{-1}$ is increasing and left-continuous with right limits, and satisfies $\alpha^{-1}(0)=0$.

The definition of $\alpha^{-1}$ and the right-continuity of $\alpha$ imply
\begin{align}
    \alpha^{-1}\circ\alpha(t)\leq t,\quad\forall t\in [0,\T]; \label{e.alpha^-1alpha(t)<t}
    \\
    s\leq\alpha\circ\alpha^{-1}(s),\quad\forall s\in[0,1]. \label{e.s<alphaalpha^-1(s)}
\end{align}

\begin{lemma}\label{l.alpha(t)=sup}
Let $\alpha$ be a p.d.f.\ on $[0,\T]$ and let $\alpha^{-1}$ be its associated quantile function. We have
\begin{align*}
    \alpha(t) = \sup\Ll\{s\in[0,1]:\: t\geq \alpha^{-1}(s)\Rr\},\quad\forall t\in[0,\T].
\end{align*}
\end{lemma}
\begin{proof}
We set $\alpha'(t)= \sup\Ll\{s\in[0,1]:\: t\geq \alpha^{-1}(s)\Rr\}$ for $t\in[0,\T]$ and we show $\alpha=\alpha'$. 
Since $\alpha^{-1}$ is left-continuous, we get $\alpha^{-1}\circ\alpha'(t)\leq t$. This along with~\eqref{e.s<alphaalpha^-1(s)} gives $\alpha'(t)\leq \alpha\circ\alpha^{-1}\circ\alpha'(t)\leq \alpha(t)$.
The definition of $\alpha'$ gives $s\leq \alpha'\circ\alpha^{-1}(s)$ for every $s$, which along with~\eqref{e.alpha^-1alpha(t)<t} implies $\alpha(t)\leq \alpha'\circ\alpha^{-1}\circ\alpha(t)\leq \alpha'(t)$.
\end{proof}

For a function $\rho$ with right limits defined on an interval $[a,b]$, we write
\begin{align}\label{e.rho(r+)=}
    \rho(r+)=\lim_{r'\downarrow r} \rho(r'),\quad\forall r \in [a,b);\qquad \rho(b+)=\rho(b).
\end{align}

\begin{lemma}\label{l....=t}
Denoting $\alpha^{-1}([0,1]) = \Ll\{\alpha^{-1}(s):\:s\in[0,1]\Rr\}$, we have
\begin{gather*}
    \alpha^{-1}\circ\alpha(t)=t,\quad \forall t\in \alpha^{-1}([0,1]);\qquad \alpha^{-1}\circ\alpha(t+)=t,\quad \forall t\in\overline{\alpha^{-1}([0,1])}.
\end{gather*}
\end{lemma}

We emphasize that $\alpha^{-1}([0,1])$ does \textit{not} denote the preimage of $[0,1]$ under $\alpha$ (the latter is the full interval $[0,T]$).
We also clarify that here and henceforth $\alpha^{-1}\circ\alpha(t+) =\lim_{t'\downarrow t}\alpha^{-1}\circ\alpha(t')$, which is in general \textit{not} equal to $\alpha^{-1}\Ll(\lim_{t' \downarrow t} \alpha(t')\Rr) = \alpha^{-1}\Ll(\alpha(t)\Rr)$ since $\alpha^{-1}\circ\alpha$ may not be right-continuous.

\begin{proof}
The first identity follows easily from~\eqref{e.alpha^-1alpha(t)<t} and~\eqref{e.s<alphaalpha^-1(s)} and the monotonicity of $\alpha$. 
We focus on the second identity. 

First, assume $t\in \alpha^{-1}([0,1])$. 
If $t=\T$, then the identity holds due to the first identity and the convention $\alpha(\T+)=\alpha(\T)$ in~\eqref{e.rho(r+)=}. Now, assume $t<\T$ and let $(t_n)_{n\in\N}$ be a decreasing sequence converging to $t$. Then, using the first identity and the first relation in~\eqref{e.alpha^-1alpha(t)<t}, we get $t= \alpha^{-1}\circ\alpha(t)\leq \alpha^{-1}\circ\alpha(t_n)\leq t_n$. Sending $n\to\infty$, we get the desired identity at $t$. 

Next, assume $t\in \overline{\alpha^{-1}([0,1])}\setminus \alpha^{-1}([0,1])$. We claim that there is a decreasing $(t_n)_{n\in\N}$ in $\alpha^{-1}([0,1])$ converging to $t$ with $t_n> t$. Suppose that this is not true. Then, there must be an increasing $(t'_n)_{n\in\N}$ in $\alpha^{-1}([0,1])$ converging to $t$. For each $t'_n$, let $s_n$ satisfy $t'_n=\alpha^{-1}(s_n)$. Then, $(s_n)_{n\in\N}$ is increasing. Denote by $s$ its limit. The left-continuity of $\alpha^{-1}$ implies $t=\alpha^{-1}(s)$ contradicting $t\not\in \alpha^{-1}([0,1])$. Hence, the claim holds and let $(t_n)_{n\in\N}$ be the sequence described. Since $(t_n)_{n\in\N}$ approaches $t$ from the right and $\alpha^{-1}\circ\alpha(t_n)=t_n$ due to the first identity, we get the second identity at $t$ by passing to the limit.
\end{proof}

Next, we recall relations between the quantile function and the probability measure it represents. We denote by $\supp \d\alpha$ the support of $\d\alpha$, defined to be the smallest closed set on which $\alpha$ has full measure. 

\begin{lemma}\label{l.law_dalpha}
The law of $\alpha^{-1}(U)$ is $\d\alpha$, where $U$ is a uniform random variable over $[0,1]$. Consequently, for every bounded measurable $h:[0,\T]\to \R$, we have that $\int_0^1 h\circ\alpha^{-1}(s)\d s = \int_0^\T h(t)\d\alpha(t)$. Moreover, $\supp \d\alpha = \overline{\alpha^{-1}((0,1])}$ and
\begin{align}
    \alpha^{-1}\circ\alpha(t+) = t,\quad \forall t \in \{0\}\cup \supp\d \alpha.\label{e.alpha^-1alpha(t)=t}
\end{align}
\end{lemma}

Here again, $\alpha^{-1}((0,1]) = \Ll\{\alpha^{-1}(s):\:s\in(0,1]\Rr\}$ is \textit{not} the preimage of $(0,1]$ under $\alpha$. 

\begin{proof}[Proof of Lemma~\ref{l.law_dalpha}]
For $t\in[0,\T]$, we have $\P\Ll(\alpha^{-1}(U)\leq t\Rr) = \sup\Ll\{s\in[0,1]:\alpha^{-1}(s)\leq t\Rr\}$, which is exactly $\alpha(t)$ due to Lemma~\ref{l.alpha(t)=sup}. It remains to identify the support of $\d\alpha$. For brevity, we write $S= \alpha^{-1}((0,1])$. Since $\P\Ll(\alpha^{-1}(U)\in\bar S\Rr)\geq \P\Ll(U\in(0,1]\Rr)=1$, we have $\supp\d\alpha \subset \bar S$. To show the other direction, let $K$ be any closed set such that $\P\Ll(\alpha^{-1}(U)\in K\Rr) = 1$. We claim that $\bar S\subset K$. Suppose otherwise $\bar S\cap K^\comple\neq \emptyset $. Since $K^\comple$ is open, we must have $S\cap K^\comple\neq \emptyset$. Let $s\in(0,1]$ satisfy $\alpha^{-1}(s) \in S\cap K^\comple$. Since $\alpha^{-1}$ is left-continuous and $K^\comple$ is open, there is $s'\in(0,s)$ sufficiently close to $s$ such that $\alpha^{-1}(r)\in S\cap K^\comple$ for every $r\in(s',s)$. Then, we have $\P\Ll(\alpha^{-1}(U)\in K^\comple\Rr)\geq \P\Ll(U\in(s',s)\Rr)>0$, reaching a contradiction. Therefore, we have $\bar S\subset K$ for every closed $K$ with full measure and thus $\bar S\subset \supp\d\alpha$. Lastly, using the characterization of $\supp\d\alpha$ and the easy observation $0=\alpha^{-1}(0)$, we can get~\eqref{e.alpha^-1alpha(t)=t} from Lemma~\ref{l....=t}.
\end{proof}

For a function $\rho:I\to \R$ defined on some interval $I$, we say that $\rho$ is \textbf{strictly increasing at $s$} if $(s,+\infty)\cap I\neq \emptyset$ and $\rho(s')>\rho(s)$ for every $s'\in (s,+\infty)\cap I$; and that $\rho$ is \textbf{strictly increasing on $J$} for some $J\subset I$ if $\rho$ is strictly increasing at every $s\in J$. If $\rho$ is strictly increasing on $I$, then we simply say that $\rho$ is \textbf{strictly increasing} as usual.

\begin{lemma}\label{l.strict_incr=>supp}
Let $\alpha:[0,\T]\to [0,1]$ be right-continuous and increasing. If $\alpha^{-1}$ is strictly increasing at some $s\in [0,1)$, then there is $t\in\{0\}\cup\supp\d\alpha$ such that $\alpha(t)=s$ and $\alpha^{-1}(s)=t$.
\end{lemma}

\begin{proof}
We first show that there exists $t\in[0,\T]$ satisfying $\alpha(t)=s$.
For $r\in\R$, set $I_r =\Ll\{t:\: r\leq \alpha(t)\Rr\}$. Then, we have $\alpha^{-1}(r)=\inf I_r$ and $I_{r'}\subset I_r$ if $r'\geq r$. Since $\alpha^{-1}$ is strictly increasing at $s$, we must have $I_s\setminus I_{\s'}\neq \emptyset$ for every $s'>s$. Hence, for every $n\in\N$, there is $t_n\in I_s\setminus I_{s+n^{-1}}$ such that $s\leq \alpha(t_n)< s+n^{-1}$. For each $n$, we set $t'_n = \min\{t_1,\dots,t_n\}$. Since $\alpha$ is increasing, we have $s\leq \alpha(t'_n)<s+n^{-1}$ and the sequence $(t'_n)_{n\in\N}$ is decreasing. Let $t$ be its limit. Since $\alpha$ is right-continuous, we send $n\to\infty$ to get $s\leq \alpha(t)\leq s$, which gives the desired $t$.

Next, we set $t=\inf\Ll\{r:\:\alpha(r)=s\Rr\}$.
Comparing with the definition of $\alpha^{-1}$ in~\eqref{e.alpha^-1=}, we have $\alpha^{-1}(s)=t$.
The first step ensures that the infimum is taken over a nonempty set. The right continuity of $\alpha$ implies $\alpha(t)=s$. 
It remains to show that if $t\neq 0$, then $t\in\supp\d\alpha$. Notice that $t<\T$ because otherwise we would have $s=\alpha(\T)=1$ which is not allowed. Hence, $t\in(0,\T)$ now.
We argue by contradiction and suppose $t\not\in\supp\d\alpha$. Since $\supp\d\alpha$ is closed, there is $\eps\in(0,t)\cap (0,\T-t)$ such that $(t-\eps,t+\eps)\not\subset \supp\d\alpha$ which means that $\alpha(t'')-\alpha(t')=\int \one_{(t',t'']}\d \alpha = 0$ for every $t-\eps<t'<t''<t+\eps$. In particular, we get $\alpha(t')=\alpha(t)=s$ for some $t'<t$, which contradicts the definition of $t$. Hence, we must have $t\in\supp\d\alpha$ which completes the proof.
\end{proof}

\begin{lemma}\label{l.exists_s_strict_inc}
If $\alpha^{-1}(u)<\alpha^{-1}(v)$, then there is $s\in [u,v)$ such that $\alpha^{-1}$ is strictly increasing at $s$.
\end{lemma}

\begin{proof}
We set $I = \Ll\{r\geq u:\:\alpha^{-1}(r) = \alpha^{-1}(u)\Rr\}$ and let $s= \sup I$. Clearly, $s\geq u$. Since $\alpha^{-1}$ is left-continuous, we have $\alpha^{-1}(s)=\alpha^{-1}(u)$. Suppose that $\alpha^{-1}$ is not strictly increasing at $s$, then there is $s'>s$ such that $\alpha^{-1}(s')= \alpha^{-1}(s)=\alpha^{-1}(u)$, contradicts the definition of $s$. Hence, $\alpha^{-1}$ is strictly increasing at $s$. Also, we must have $s<v$ because otherwise we have $\alpha^{-1}(u)< \alpha^{-1}(v)\leq \alpha^{-1}(s)= \alpha^{-1}(u)$, which is absurd. 
\end{proof}

\begin{lemma}\label{l.equival_jumps}
Let $s\in[0,1)$. The following are equivalent:
\begin{enumerate}
    \item it holds that $\alpha^{-1}(s+)>\alpha^{-1}(s)$;
    \item there are $t,t_\star\in \{0\}\cup\supp\d\alpha$ such that
    \begin{align}\label{e.t,t_star_prop}
        t_\star>t,\qquad \alpha(t)=s,\qquad t=\alpha^{-1}(s),\qquad t_\star =\alpha^{-1}(s+),\qquad (t,t_\star)\cap \supp\d \alpha=\emptyset.
    \end{align}
\end{enumerate}

\end{lemma}

\begin{proof}
The second statement clearly implies the first one. We focus on showing the other direction. The first statement implies that $\alpha^{-1}$ is strictly increasing at $s$. Let $t$ be given as in Lemma~\ref{l.strict_incr=>supp} and we have $\alpha(t)=s$ and $t=\alpha^{-1}(s)$.

Then, we show the existence of $t_\star$. Suppose that there is a decreasing sequence $(t_n)_{n\in\N}$ in $\supp\d\alpha$ converging to $t$. Since $\alpha$ is right-continuous, setting $s_n =\alpha(t_n)$, we have $\lim_{n\to\infty} s_n = s$. Using~\eqref{e.alpha^-1alpha(t)=t} and the convergence of $(t_n)_{n\in\N}$, we have $\lim_{n\to\infty} \alpha^{-1}(s_n) = \alpha^{-1}(s)$, which contradicts the assumption of the first statement. Therefore, setting $t_\star=\inf\Ll\{t'>t:\:t'\in\supp\d\alpha\Rr\}$, we must have $t_\star>t$. Since $\supp\d\alpha$ is closed, we also have $t_\star\in\supp\d\alpha$. Lastly, it is clear from the definition that $(t,t_\star)\cap \supp\d\alpha=\emptyset$.

It remains to verify $t_\star = \alpha^{-1}(s+)$.
First, we show
\begin{align}\label{e.alpha-1(s')>t_star}
    \alpha^{-1}(s')\geq t_\star,\quad\forall s'>s.
\end{align}
We argue by contradiction and suppose $\alpha^{-1}(s')>t_\star$ for some $s'>s$. Due to $\alpha^{-1}\circ\alpha(t_\star)=t_\star$ by~\eqref{e.alpha^-1alpha(t)=t}, we have $\alpha^{-1}(s')<\alpha^{-1}(\alpha(t_\star))$ and thus we must have $s'<\alpha(t_\star)$. By Lemma~\ref{l.exists_s_strict_inc} and then Lemma~\ref{l.strict_incr=>supp}, there is $s''\in [s',\alpha(t_\star))$ and $t''\in\supp \d\alpha$ such that $s''=\alpha(t'')$. Then, we must have $t''<t_\star$ and thus
\begin{align*}
    t=\alpha^{-1}(s)<\alpha^{-1}(s')\leq \alpha^{-1}(s'') \stackrel{\eqref{e.alpha^-1alpha(t)=t}}{=}t'' <t_\star.
\end{align*}
In particular, we have $t''\in (t,t_\star)\cap \supp\d\alpha$ and reaches a contradiction. Hence, \eqref{e.alpha-1(s')>t_star} holds.

Due to $t_\star > t$, we have $\alpha(t_\star)\geq \alpha(t)=s$. Also, recall $\alpha^{-1}\circ\alpha(t_\star) =t_\star$ due to~\eqref{e.alpha^-1alpha(t)=t}.
These along with~\eqref{e.alpha-1(s')>t_star} imply $t_\star = \inf_{s'>s}\alpha^{-1}(s')$. Since $\alpha^{-1}$ is increasing, we deduce that $t_\star = \alpha^{-1}(s+)$.
\end{proof}

\subsection{Decompositions}

\subsubsection{Decomposition of one path}

Given $q\in\mcl Q_\infty$ which is a function on $[0,1)$, we denote by $\aq:[0,1]\to S^\D_+$ its left-continuous version defined by
\begin{align}\label{e.aq=}
    \aq(0)=0;\qquad \aq(s)=\lim_{u\uparrow s} q(u),\quad\forall s\in(0,1].
\end{align}
For some $T>0$, a pair $(L,\alpha)$ is said to be a \textbf{decomposition of $q$ (defined on $[0,\T]$)} if $L:[0,\T]\to \S^\D_+$ is Lipschitz and increasing, $\alpha:[0,\T]\to[0,1]$ is a p.d.f.,\ and 
\begin{align}\label{e.aq=Lalpha(0,1]}
    \aq(s) = L\circ \alpha^{-1}(s),\quad\forall s \in (0,1].
\end{align}
It is easy to see that decompositions of $q$ are not unique. In particular, given $\aq$ and $\alpha$, one can only determine $L$ on $\supp\d\alpha$, more precisely,
\begin{align}\label{e.L(t)=qalpha(t+)}
     L (t) = \aq \circ \alpha(t+),\quad\forall t\in \supp\d \alpha.
\end{align}
Indeed, fix any $t\in \supp \d\alpha$, we must have $\alpha(t')>0$ for every $t'>t$. Then, using~\eqref{e.aq=Lalpha(0,1]}, we have $\aq \circ \alpha(t') = L\circ\alpha^{-1}\circ\alpha(t')$ for all $t'>t$ and thus $\aq \circ \alpha(t+) = L\circ\alpha^{-1}\circ\alpha(t+)$. Then, \eqref{e.L(t)=qalpha(t+)} follows from this and~\eqref{e.alpha^-1alpha(t)=t} in Lemma~\ref{l.law_dalpha}. 

Before proceeding, we comment that the value of $\aq(0)$ is insignificant.\footnote{Alternatively, one can set $\aq(0)=\aq(0+)$, which is similar to the choice of $q(1)$ in~\eqref{e.q(1)=}.}
We make the choice in~\eqref{e.aq=} so that when we choose $\alpha$ to satisfy $\tr \aq =\alpha^{-1}$ in the later construction, we indeed have $\tr\aq(0)= \alpha^{-1}(0)=0$. It is also important to notice that the relation in~\eqref{e.aq=Lalpha(0,1]} is not required to hold at $s=0$. Otherwise, we would have $L(0)=\aq(0)$, which is an unwanted restriction beyond~\eqref{e.L(t)=qalpha(t+)} if $0\not\in \supp\d\alpha$.

This decomposition is needed to express the integration of the cascade in terms of the solution to the so-called Parisi PDE. The coefficients of the second-order and first-order terms in this PDE are determined by $L$ and $\alpha$, respectively (see~\eqref{e.Parisi_PDE}).

We comment that taking the left-continuous version $\aq$ is preferred here because $\aq$ resembles a quantile function and thus enjoys better properties due to the duality between quantile functions and p.d.f.s. Since $q$ is right-continuous with left limits, taking $\aq$ does not lose information and we can always recover $q$ from $\aq$ by taking the right-continuous version.

Since there is too much freedom in choosing a decomposition, sometimes, we need to put an extra restriction.
A decomposition $(L,\alpha)$ of some $q$ is said to be \textbf{pinned} if $L(0)=0$.

\subsubsection{Canonical decompositions}

We turn to the existence of such a decomposition.
In the following, we describe the construction of an arguably \textit{canonical} decomposition of $\aq$. 
Let us start by explaining the motivation.

First, we want this decomposition to reflect the phenomenon of synchronization of overlaps in \cite{pan.potts,pan.vec}. 
To explain this, let us denote by $R$ a random variable whose law is the limiting distribution of $\frac{\sigma \sigma'^{\intercal}}{N}$ under $\mathbb{E}\langle \cdot \rangle_N^{\mathrm{pert}}$ as $N \to \infty$, where $\sigma$ and $\sigma'$ are independent samples drawn from $\mathbb{E}\langle \cdot \rangle_N^{\mathrm{pert}}$. 
Here $\langle \cdot \rangle_N^{\mathrm{pert}}$ denotes the Gibbs measure $\langle \cdot \rangle_N$ in~\eqref{e.<>_N=} with a suitably chosen perturbation
added.
Panchenko's synchronization principle in \cite{pan.potts,pan.vec} states that there exists some Lipschitz path $L$ such that $R=L(\tr(R))$ a.s. If we denote by $\alpha$ the probability distribution function of the law of $\tr(R)$, then we have $\tr(R)\stackrel{\d}{=} \alpha^{-1}(U)$ where $U$ is the uniform random variable on $[0,1]$. Setting $\aq = L\circ\alpha^{-1}$, we have $R\stackrel{\d}{=} \aq(U)$. Then, we should expect $\tr \aq = \alpha^{-1}$ which is enough to determine $\alpha$.
After fixing $\alpha$, the value of $L$ on $\{0\}\cup\supp\d\alpha$ is completely determined via~\eqref{e.L(t)=qalpha(t+)}. There is freedom in choosing its value outside $\{0\}\cup\supp\d\alpha$ and we simply use linear interpolations. 

With the above explained, we are ready to give the definition. Given $q\in \mcl Q_\infty$, we call $(L,\alpha)$ the \textbf{canonical decomposition of $q$} if $\alpha:[0,\T]\to [0,1]$ is a p.d.f.\ satisfying $\alpha^{-1}=\tr \aq$ with $\T=\tr \aq(1)$ and $L:[0,\T]\to\S^\D_+$ is given by
\begin{align}\label{e.L=linear}
\begin{split}
    L (t)  &= \overrightarrow{q} \circ \alpha(t+),\quad\forall t\in\{0\}\cup\supp\d \alpha;
    \\
    L(t) & = \frac{t^\mathrm{r}-t}{t^\mathrm{r}-t^\mathrm{l}}L\Ll(t^\mathrm{l}\Rr) + \frac{t-t^\mathrm{l}}{t^\mathrm{r}-t^\mathrm{l}}L\Ll(t^\mathrm{r}\Rr),\quad\forall t\in (0,\T]\setminus \supp\d\alpha,
\end{split}
\end{align}
where for every $t\in[0,\T]$ we set
\begin{align}\label{e.t^lt^r}
    t^\mathrm{l}= \sup\Ll\{r\in\supp\d\alpha:\:r\leq t\Rr\},\qquad t^\mathrm{r}= \inf\Ll\{r\in\supp\d\alpha:\:r\geq t\Rr\}.
\end{align}
When the infimum in $t^\mathrm{r}$ is taken over an empty set, we understand $t^\mathrm{r}=+\infty$ and $L(t) = L(t^\mathrm{l})$ in~\eqref{e.L=linear}.

Notice that in addition to values of $L$ in~\eqref{e.L(t)=qalpha(t+)} that are determined by $\aq$ and $\alpha$, we prescribe the value of $L$ at $0$ in~\eqref{e.L=linear}.

\begin{lemma}\label{l.decomp_q}
For any $q\in \mcl Q_\infty$, let $(L,\alpha)$ be the canonical decomposition of $q$ described above. Then, $(L,\alpha)$ is a pinned decomposition of $q$.
Moreover, $\tr L(t) = t$ for every $t\in[0,\T]$ and $\|L\|_\mathrm{Lip}\leq\sqrt{\D}$.
\end{lemma}

For a canonical decomposition $(L,\alpha)$, we can say that $L$ has unit speed due to $\tr L(t)=t$.

\begin{proof}[Proof of Lemma~\ref{l.decomp_q}]
We verify that $L$ is increasing.
For convenience, we start by observing that we can extend the identity on the second line of \eqref{e.L=linear} to every $t \in [0,T]$ with the understanding that, when $t \in \supp \d \alpha$ and thus $t=t^\mathrm{l}=t^\mathrm{r}$,  both fractions in~\eqref{e.L=linear} are set to be $1$.
Fix any $t\leq {t'}$ and let $t^\mathrm{l},t^\mathrm{r},{t'}^\mathrm{l},{t'}^\mathrm{r}$ be given as in~\eqref{e.t^lt^r}. Due to their definitions, we can see that there are two possible cases: $t^\mathrm{l}={t'}^\mathrm{l}$ and $t^\mathrm{r}={t'}^\mathrm{r}$; or $t^\mathrm{r}\leq {t'}^\mathrm{l}$.\footnote{Indeed, if the first case does not hold, then either we have $t^\mathrm{l}<{t'}^\mathrm{l}$ or $t^\mathrm{r}<{t'}^\mathrm{r}$. In the former case, if ${t'}^\mathrm{l}<t^\mathrm{r}$, we must have ${t'}^\mathrm{l}<t$ (otherwise we have $t\leq {t'}^\mathrm{l}<t^\mathrm{r}$ contradicting the definition of $t^\mathrm{r}$). But, then we have $t^\mathrm{l}<{t'}^\mathrm{r}<t$ which contradicts the definition of $t^\mathrm{l}$. Hence, we are only left with the possibility $t^\mathrm{r}<{t'}^\mathrm{r}$. Then, we must have ${t'}>t^\mathrm{r}$ (otherwise we have ${t'}\leq t^\mathrm{r}<{t'}^\mathrm{r}$ contradicts the definition of ${t'}^\mathrm{r}$). Now, ${t'}>t^\mathrm{r}$ implies $t^\mathrm{r}\leq {t'}^\mathrm{l}$.} In both cases, since $\aq\circ\alpha$ is increasing, we can get from~\eqref{e.L=linear} that $L(t)\leq L({t'})$.

Taking the trace on both sides of the expressions of $L$ in~\eqref{e.L=linear}, and using $\alpha^{-1}=\tr\aq$ and~\eqref{e.alpha^-1alpha(t)=t}, we can see that $\tr L(t) =t$ for every $t\in [0,\T]$.

To show that $L$ is Lipschitz, we need the following basic result.
For any $a\in \S^\D_+$, by diagonalizing $a$, we have $|a| = \sqrt{\sum_{i=1}^\D \lambda_i^2}$ where $(\lambda_i)_{i=1}^\D$ are eigenvalues of $a$, which implies
\begin{align}\label{e.tr(a)sim|a|}
    D^{-\frac{1}{2}}\tr(a) \leq |a|\leq D^\frac{1}{2}\tr(a),\quad\forall a\in \S^\D_+.
\end{align}
Using that $L$ is increasing, \eqref{e.tr(a)sim|a|}, and $\tr L(t)=t$ for $t\in[0,\T]$, we have
\begin{align*}
    \Ll|L(t)-L(t')\Rr|\leq D^\frac{1}{2}\Ll|\tr L(t)-\tr L(t')\Rr| = D^\frac{1}{2}\Ll|t-t'\Rr|
\end{align*}
which verifies that $L$ is Lipschitz with coefficient bounded by $\sqrt{D}$.

We verify $\aq = L\circ\alpha^{-1}$ on $(0,1]$. 
Using~\eqref{e.s<alphaalpha^-1(s)} and the monotonicity of $\aq$, we have $\aq\circ\alpha\circ\alpha^{-1}(s)\geq \aq(s)$ for every $s$. Then,
\begin{align*}
    \tr\Ll(\aq\circ\alpha\circ\alpha^{-1}(s) - \aq(s)\Rr) = \alpha^{-1}\circ\alpha\circ\alpha^{-1}(s)- \alpha^{-1}(s)\stackrel{\text{L.\ref{l....=t}}}{=}0.
\end{align*}
Now, due to~\eqref{e.tr(a)sim|a|}, we must have $\aq(s)=\aq\circ\alpha\circ\alpha^{-1}(s)$ for every $s$. It remains to show $L\circ\alpha^{-1}(s)= \aq\circ\alpha\circ\alpha^{-1}(s)$ for every $s\in (0,1]$. Set $t=\alpha^{-1}(s)$. By the first line in~\eqref{e.L=linear} and the characterization of $\supp\d\alpha$ in Lemma~\ref{l.law_dalpha}, we have $L\circ\alpha^{-1}(s) = \aq\circ\alpha\Ll(t+\Rr)$. By monotonicity, we have $\aq\circ\alpha\Ll(t+\Rr)\geq \aq\circ\alpha\Ll(t\Rr)$. Taking the trace, we get
\begin{align*}
    \tr\Ll( \aq\circ\alpha\Ll(t+\Rr)-\aq\circ\alpha\Ll(t\Rr)\Rr) = \alpha^{-1}\circ\alpha(t+) - \alpha^{-1}\circ\alpha(t) \stackrel{\text{L.\ref{l....=t}\,\&\,\ref{l.law_dalpha}}}{=}t -t=0,
\end{align*}
which by~\eqref{e.tr(a)sim|a|} implies $\aq\circ\alpha\Ll(t+\Rr)=\aq\circ\alpha\Ll(t\Rr)$ and thus $L\circ\alpha^{-1}(s)=\aq\circ\alpha\circ\alpha^{-1}(s)$. As commented earlier, this verifies $\aq = L\circ\alpha^{-1}$ on $(0,1]$. Hence, we conclude that $(L,\alpha)$ is a decomposition of $q$.
\end{proof}

\subsubsection{Joint decomposition of multiple paths}\label{s.joint_decomp}

For $n\in\N$, given $q_1,\dots ,q_n \in \mcl Q_\infty$, we define a new path $\boldsymbol{q}:[0,1]\to \S^{n\D}_+$ by setting $\boldsymbol{q}(s) = \mathrm{diag} \Ll(q_1(s),\dots ,q_n(s)\Rr)$ for every $s\in[0,1]$.
We define its left-continuous version similarly as $\overrightarrow{\boldsymbol{q}}(s) = \lim_{u\uparrow s}\boldsymbol{q}(s)$ for $s\in(0,1]$ and $\overrightarrow{\boldsymbol{q}}(0)=0$.
A tuple $(L_1,\dots ,L_n,\alpha)$ is said to be a \textbf{joint decomposition} of $q_1,\dots,q_n$ if $(L_k,\alpha)$ is a decomposition of $q_k$ for every $k\in\{1,\dots,n\}$ for a common p.d.f.\ $\alpha:[0,\T]\to[0,1]$. 
Defining $\boldsymbol{L}(t) = \mathrm{diag}\Ll(L_1(t),\dots,L_n(t)\Rr)$, we can see that this is equivalent to that $(\boldsymbol{L},\alpha)$ is a decomposition of $\boldsymbol{q}$ (with the ambient matrix space $\S^{n\D}_+$). Similarly, a joint decomposition $(\boldsymbol{L},\alpha)$ is said to be \textbf{pinned} if $\boldsymbol{L}(0)=0$.

We can construct the canonical decomposition similarly. A tuple $(L_1,\dots ,L_n,\alpha)$ is said to be the \textbf{canonical joint decomposition} of $q_1,\dots,q_n$ if $(\boldsymbol{L},\alpha)$ is the canonical decomposition of $\boldsymbol{q}$. Then, $(\boldsymbol{L},\alpha)$ enjoys the properties generalized in the obvious way (i.e.\ with $\D$ replaced by $n\D$) in Lemma~\ref{l.decomp_q}. 
The canonical joint decomposition is given explicitly by
\begin{align}\label{e.alpha=_joint_decomp}
    \alpha^{-1}(s) = \sum_{k=1}^n \tr \overrightarrow{q_k}(s),\quad\forall s \in[0,1],
\end{align}
and for every $k\in\{1,\dots,n\}$,
\begin{align}\label{e.L_k=_joint_decomp}
\begin{split}
    L_k (t)  &= \overrightarrow{q_k} \circ \alpha(t+),\quad\forall t\in\{0\}\cup\supp\d \alpha;
    \\
    L_k(t) &= \frac{t^\mathrm{r}-t}{t^\mathrm{r}-t^\mathrm{l}}L_k\Ll(t^\mathrm{l}+\Rr) + \frac{t-t^\mathrm{l}}{t^\mathrm{r}-t^\mathrm{l}}L_k\Ll(t^\mathrm{r}+\Rr),\quad\forall t\in (0,\T]\setminus \supp\d\alpha
\end{split}
\end{align}
for $t^\mathrm{l}$ and $t^\mathrm{r}$ in~\eqref{e.t^lt^r}.

\section{The Parisi PDE}
\label{s.parisi.pde}

In this section, we recall basic properties of the Parisi PDE.
We fix some $\T>0$ throughout this section.
For a p.d.f.\ $\alpha$ on $[0,\T]$, a Lipschitz increasing path $L:[0,\T]\to\S^\D_+$, and a smooth function $\phi:\R^\D\to\R$ with bounded derivatives, we consider the Parisi PDE
\begin{gather}\label{e.Parisi_PDE}
    \partial_t \Phi(t,x) + \la \dot L(t)\,,\,\nabla^2\Phi(t,x)+\alpha(t)\nabla \Phi\nabla \Phi^\intercal(t,x)\ra_{\S^D}  =0, 
\end{gather}
for $(t,x)\in [0,\T]\times \R^\D$ with terminal condition $\Phi(T,\cdot) = \phi$. The notation $\dot L$ stands for the time derivative of $L$, which is defined almost everywhere since $L$ is Lipschitz. 
For our purposes, we consider the following class of initial conditions
\begin{align}\label{e.phi=initial}
    \phi(x) = \log\int \exp\left(\sigma\cdot x - \sigma\sigma^\intercal\cdot L(T)\right) \d \mu(\sigma),\quad\forall x\in \R^\D,
\end{align}
where $\mu$ is a (positive) finite measure compactly supported on $\R^\D$.

We describe a probabilistic way of constructing the solution to \eqref{e.Parisi_PDE}. We start by solving the equation explicitly for smooth $L$ and a step function $\alpha$. We denote by $\mcl M$ the collection of p.d.f.s on $[0,\T]$ and by $\mcl M_\d$ its subcollection consisting of $\alpha$ of the form:
\begin{align}\label{e.alpha=}
    \alpha = \sum_{l=0}^{K}(s_{l}-s_{l-1})\one_{[t_l,\infty)} = \sum_{l=0}^{K-1} s_l \one_{[t_l,t_{l+1})}  + s_K \one_{\{t_K\}}
\end{align}
where $s_{-1},s_0,\dots,s_{K}\in[0,1]$ and $t_0,t_1,\dots,t_K\in [0,\T]$ satisfy
\begin{align*}0=s_{-1}\leq s_0< \dots<s_{K} = 1,\qquad 
    0= t_0\leq t_1<\dots< t_K=T.
\end{align*}
Since for $\alpha\in\mcl M_\d$ we have $\alpha(t)\leq s_{K-1}<1$ for every $t<T$ and $\alpha(T)=1$, we get
\begin{align}\label{e.alpha^-1=T_discrete}
    \alpha^{-1}(1)=T,\qquad\forall \alpha\in\mcl M_\d.
\end{align}
In other words, $T$ is the right endpoint of $\supp\d \alpha$ whenever $\alpha\in \mcl M_\d$.

\begin{definition}[Parisi PDE solution for smooth $L$ and discrete $\alpha$]\label{d.sol_pde_discrete}

For every smooth increasing path $L:[0,\T]\to\S^\D_+$ and $\alpha \in \mcl M_\d$, the \textbf{Parisi PDE solution} $\Phi=\Phi_{\mu,L,\alpha}$ is defined as follows:
\begin{itemize}
    \item Set $\Phi(\T,\cdot) = \phi$ as in~\eqref{e.phi=initial}.
    \item Inductively, for every $l\in \{1,\cdots,K\}$, $t\in [t_{l-1},t_l)$ and $x\in\R^\D$, define
    \begin{align}\label{e.Phi(s,x)}
        \Phi(t,x) = \frac{1}{s_{l-1}}\log\E\exp\left(s_{l-1} \Phi\left(t_l,\,\sqrt{2L(t_l)-2L(t)}g_l+x\right)\right)
    \end{align}
    where $(g_l)_{l\in\{0,\dots,K\}}$ are independent standard $\R^\D$-valued Gaussian vectors.
\end{itemize}
\end{definition}

For $l = 1$, the right-hand side of~\eqref{e.Phi(s,x)} is understood to be $\E \Phi(t_1, \sqrt{2L(t_1)-2L(t)}g_1+x)$ if $s_0=0$.

It is well-known that the Parisi PDE solution for discrete $\alpha$ satisfies the equation in the classical sense except at discontinuity points of $\alpha$.

\begin{lemma}\label{l.PDE}
Let $\Phi = \Phi_{\mu,L,\alpha}$ for a smooth increasing path $L:[0,\T]\to\S^\D_+$ and $\alpha \in \mcl M_\d$ as in~\eqref{e.alpha=}.
For every $l\in\{1,\dots,K\}$ and at every $(t,x)\in (t_{l-1},t_l)\times \R^\D$, the function $\Phi$ is differentiable and equation~\eqref{e.Parisi_PDE} is satisfied.
\end{lemma}

We refer to \cite[Lemma~2.4]{chen2023parisi} for a proof, which uses the Hopf--Cole transformation and a direct computation.

We equip $\mcl M$ with the metric induced by $L^1$-norm.
For any two normed spaces $\mcl X$ and $\mcl Y$, let $C(\mcl X;\mcl Y)$ be the collection of $\mcl Y$-valued continuous functions on $\mcl X$ equipped with the uniform norm.
We extend Definition~\ref{d.sol_pde_discrete} to general $L$ and $\alpha$. 
A sequence $((L_n,\alpha_n))_{n\in\N}$ is said to converge to $(L,\alpha)$ if $(\alpha_n)_{n\in\N}$ converges to $\alpha$ in $\mcl M$ and $(L_n)_{n\in\N}$ of converges to $L$ in $C([0,\T];\S^\D_+)$.

\begin{definition}[Parisi PDE solution]\label{d.sol_pde}

Given a Lipschitz function $L$, $\alpha\in\mcl M$, and an initial condition $\phi$ in~\eqref{e.phi=initial}, the associated \textbf{Parisi PDE solution} $\Phi_{\mu,L,\alpha}$ is the limit in $C([0,\T]\times \R^\D;\R)$ 
of $\Ll(\Phi_{\mu,L_n,\alpha_n}\Rr)_{n\in\N}$ (given in Definition~\ref{d.sol_pde_discrete}) for any sequence $((L_n,\alpha_n))_{n\in\N}$ converging to $(L,\alpha)$ where $L_n$ is smooth and $\alpha_n\in \cM_\d$ for each $n$.
\end{definition}

\begin{lemma}
The Parisi PDE solution in Definition~\ref{d.sol_pde} is well-defined and independent of the approximation sequences. 
\end{lemma}

\begin{proof}
First, we show the existence of approximation sequences. We can approximate $\alpha^{-1}$ by step functions $\alpha^{-1}_n$ satisfying $\alpha^{-1}_n(1)=\alpha^{-1}(1)$ and then take right-continuous inverses to get $\alpha_n$. This gives a sequence approximating $\alpha$ in $\mcl M$. Next, let $\Lambda:\R_+\to\R_+$ be a smooth function compactly supported on $(0,1)$ satisfying $\int \Lambda =1$. For each $\eps>0$, we take $\Lambda_\eps = \frac{1}{\eps}\Lambda\Ll(\frac{\cdot}{\eps}\Rr)$. We extend $L$ to be defined on $\R$ by setting $L(t)=L(0)$ for $t<0$ and $L(t)=L(\T)$ for $t>T$, and we set 
\begin{align}\label{e.L_n=}
	L_n (t) = \int \Lambda_{1/n}(t') L (t-t')\d t',\quad\forall t\in[0,\T].
\end{align}
It is easy to see that $L_n$ is still increasing and Lipschitz with $\|L_n\|_\mathrm{Lip}\leq \|L\|_\mathrm{Lip}$. Moreover, $(L_n)_{n\in\N}$ converges to $L$ uniformly. 

Next, we recall a representation of the Parisi PDE solution associated with smooth $L$ and discrete $\alpha$ \cite[Lemma~2.7]{chen2023parisi}:
\begin{align}\label{e.Phi=Ecascade_discrete}
    \Phi_{\mu,L,\alpha}(t,x) = \E \log \iint \exp\Ll(\sqrt{2}\sigma\cdot w^{\pi_t}(\brho)+\sigma\cdot x -\sigma\sigma^\intercal \cdot L(T)\Rr)\d \mu(\sigma)\d \mathfrak{R}(\brho)
\end{align}
for every $(t,x)\in [0,\T]\times\R^\D$,
where $\pi_t = L\circ \alpha^{-1}_{[t}(\cdot) - L(t)$  and $\alpha_{[t}=\alpha \one_{[t,\T]}$.
Using the representation~\eqref{e.Phi=Ecascade_discrete}, we want to obtain estimates of Parisi PDE solutions in terms of $L$ and $\alpha$. 
Due to $\alpha_{[t}\in\mcl M_\d$ and~\eqref{e.alpha^-1=T_discrete}, we have $\pi_t(1) = L(T)-L(t)$ for every $t$, which allows us to rewrite~\eqref{e.Phi=Ecascade_discrete} into
\begin{align}\label{e.Phi=f_discrete}
    \Phi_{\mu,L,\alpha}(t,x) = \mathbf{f}_\mu(\pi_t,x,-L(t)),\quad\forall (t,x)\in[0,T]\times \R^\D,\ \alpha\in \mcl M_\d.
\end{align}
Let $L'$ be smooth and $\alpha'\in\mcl M_\d$. 
For any fixed $t$, we define $\pi'_t$ and $\alpha'_{[t}$ similarly, and we have $\Phi_{\mu,L',\alpha'}=\mathbf{f}_\mu(\pi'_t,x,-L'(t))$.
Then, by Corollary~\ref{c.interpolation}, we get
\begin{align}\label{e.|Phi_L-Phi_L'|<}
\begin{split}
    \Ll|\Phi_{\mu,L,\alpha}(t,x)-\Phi_{\mu,L',\alpha'}(t,x)\Rr|\leq \Ll|L(t)-L'(t)\Rr| + \int_0^1 \Ll|\pi_t(s)-\pi'_t(s)\Rr|\d s
    \\
    \leq 3\Ll\|L-L'\Rr\|_{L^\infty} + \Ll\|L\Rr\|_\mathrm{Lip}\Ll\|\alpha-\alpha'\Rr\|_{L^1},
\end{split}
\end{align}
where $L$ and $L'$ are smooth and $\alpha,\alpha'\in \mcl M_\d$. In the last inequality, we also used $\Ll\|\alpha^{-1}-\alpha'^{-1}\Rr\|_{L^1} = \Ll\|\alpha-\alpha'\Rr\|_{L^1}$ allowed by Fubini's theorem.

Now, let $L$ be a general Lipschitz increasing function and $\alpha$ be a general p.d.f. Given approximating sequences $(L_n)_{n\in\N}$ and $(\alpha_n)_{n\in\N}$, we can use the above estimate to see that $\Phi_{\mu,L_n,\alpha_n}$ converges uniformly to some limit $\Phi_{\mu,L,\alpha}$ in the space of $C([0,\T]\times \R^\D;\R)$. The above estimate also implies the independence of choice of approximating sequences.
\end{proof}

\begin{lemma}[Cascade representation of the solution]\label{l.cascade_rep}
Let $\Phi_{\mu,L,\alpha}$ be a Parisi PDE solution. For every $(t,x)\in [0,\T]\times \R^\D$, we have
\begin{align}\label{e.Phi=Ecascade}
    \Phi_{\mu,L,\alpha}(t,x) = \E \log \iint \exp\Ll(\sqrt{2}\sigma\cdot w^{\pi_t}(\brho)+\sigma\cdot x - \sigma\sigma^\intercal\cdot(\pi_t(1)+L(t))\Rr)\d \mu(\sigma)\d \mathfrak{R}(\brho)
\end{align}
where $\pi_t = L\circ \alpha^{-1}_{[t}(\cdot) - L(t)$ and $\alpha_{[t}=\alpha \one_{[t,\T]}$. In other words, we have $\Phi_{\mu,L,\alpha}(t,x) =\mathbf{f}_\mu(\pi_t,x,-L(t))$ for every $(t,x)\in[0,\T]\times\R^\D$, where $\mathbf{f}_\mu$ is given as in~\eqref{e.f(q),<>_q=}.
\end{lemma}

Given this lemma, we can directly define the Parisi PDE solution via~\eqref{e.Phi=Ecascade}. Here, $\pi_t$ is a left-continuous path with right limits but results in Section~\ref{s.cascade} are stated for right-continuous paths. This is not an issue due to Remark~\ref{r.left-cts}.

\begin{proof}We approximate $\Phi_{\mu,L,\alpha}$ as described in Definition~\ref{d.sol_pde}.
Then, this lemma follows from the representation \eqref{e.Phi=f_discrete} for approximations and Corollary~\ref{c.interpolation} (to derive a similar estimate as in~\eqref{e.|Phi_L-Phi_L'|<} to bound the discrepancy between the right-hand side of~\eqref{e.Phi=Ecascade} and its approximations).
\end{proof}

\begin{remark}[Extension of Parisi PDE]
We can define $\Phi_{\mu,L,\alpha}$ on $\R_+\times \R^\D$ through~\eqref{e.Phi=Ecascade}. Using the observation that $\alpha^{-1}_{[t}(1)=t$ and thus $\pi_t=0$ for every $t\geq \alpha^{-1}(1)$, we can see that
\begin{align*}
    \Phi_{\mu,L,\alpha}(t,x) = \log \int\exp\Ll(\sigma\cdot x - \sigma\sigma^\intercal\cdot L(t) \Rr) \d \mu(\sigma),\quad\forall (t,x)\in[\alpha^{-1}(1),\infty)\times \R^\D,
\end{align*}
which trivially satisfies the equation~\eqref{e.Parisi_PDE} on $[\alpha^{-1}(1),\infty)\times \R^\D$ and the condition in~\eqref{e.phi=initial} at $t=T>\alpha^{-1}(1)$.
\end{remark}

Let $\mathscr{I} = \cup_{k\in\{0\}\cup \N} \{1,\dots,D\}^k$ be the set of multi-indices.
For $\bi = (i_1,i_2,\dots,i_k)\in \{1,\dots,D\}^k$ for some $k\in\N$, we set $|\bi|=k$ and $\partial^\bi_x = \partial_{x_{i_1}}\partial_{x_{i_2}}\cdots \partial_{x_{i_k}}$. We set $\partial^\emptyset$ to be the identity operator. For $k\in\N$, we write $\nabla^k = \Ll(\partial^\bi_x\Rr)_{\bi:\:|\bi|=k}$.

\begin{proposition}[Regularity of the solution]\label{p.reg_Phi}
Let $\Phi = \Phi_{\mu,L,\alpha}$ be any Parisi PDE solution.
The following holds:
\begin{enumerate}
    \item \label{i.d_xPhi_exists} For every $\bi\in\mathscr{I}$, the real-valued function $\partial^\bi_x\Phi$ exists and is continuous on $[0,\T]\times \R^\D$.
    \item \label{i.dxPhi_cvg} For every $\bi\in\mathscr{I}$ (including $\bi=\emptyset$), there is a constant $C_\bi>0$ such that for any two Parisi PDE solutions $\Phi_{\mu,L,\alpha}$ and $\Phi_{\mu,L',\alpha'}$, it holds that
    \begin{align*}
        \sup_{(t,x)\in [0,T]\times \R^\D}\Ll|\partial_x^\bi\Phi_{\mu,L,\alpha}(t,x) - \partial_x^\bi\Phi_{\mu,L',\alpha'}(t,x)\Rr|\leq C_\bi\Ll(3 \Ll\|L-L'\Rr\|_{L^\infty} + \Ll\|L\Rr\|_\mathrm{Lip}\Ll\|\alpha-\alpha'\Rr\|_{L^1} \Rr).
    \end{align*}
    \item \label{i.sup|d_xPhi|<infty} For every $\bi\in\mathscr{I}$ with $|\bi|\geq 1$, 
    \begin{align*}
        \sup_{(L,\alpha)}\sup_{[0,\T]\times \R^\D}\Ll|\partial^\bi_x\Phi\Rr|<\infty. \end{align*}
    In particular, we have $|\nabla \Phi|\leq \sup_{\sigma \in\supp\mu}|\sigma|$ everywhere on $[0,T]\times \R^D$.
    \item \label{i.dsdxPhi_exists} 
    For every $\bi\in\mathscr{I}$, $\partial_t \partial^\bi_x\Phi$ exists and is bounded almost everywhere on $[0,T]\times \R^\D$ and thus $\Phi$ satisfies the Parisi PDE~\eqref{e.Parisi_PDE} almost everywhere. 
    
    In particular, letting $\mathsf{disc}_\alpha$ be the set of discontinuity points of $\alpha$, if $L$ is continuously differentiable, then $\partial_t\partial^\bi_x\Phi$ exists and is continuous on $\Ll([0,\T]\setminus \mathsf{disc}_\alpha\Rr)\times \R^\D$ for every $\bi\in\mathscr{I}$ and $\Phi$ satisfies~\eqref{e.Parisi_PDE} everywhere on $\Ll([0,\T]\setminus \mathsf{disc}_\alpha\Rr)\times \R^\D$.
\end{enumerate}
\end{proposition}

The set $\mathsf{disc}_\alpha$ is countable due to the monotonicity of $\alpha$. When $L$ is continuously differentiable and $\alpha$ is continuous, the last part implies that $\Phi$ is a classical solution.

\begin{proof}
Using the representation in Lemma~\ref{l.cascade_rep} we can inductively verify that, for every $(t,x)$ and every $\bi$ with $|\bi| \ge 1$, there is a polynomial $F_\bi : \Ll(\R^\D\Rr)^{|\bi|} \to \R$ independent of $(L,\alpha)$ such that
\begin{align}\label{e.d_xPhi=E<F>}
    \partial_x^\bi\Phi(t,x) = \E \la F_\bi \Ll(\sigma^1,\sigma^2,\dots \sigma^{|\bi|}\Rr) \ra_{\mu,\pi_t,x,-L(t)}
\end{align}
where $\la\cdot \ra_{\mu,\pi_t,x,-L(t)}$ is the Gibbs measure given as in~\eqref{e.f(q),<>_q=} and each $\sigma^l$ is an independent copy of $\sigma$ under $\la\cdot \ra_{\mu,\pi_t,x,-L(t)}$. This along with the assumption on the support of $\mu$ gives the main display in Part~\eqref{i.sup|d_xPhi|<infty}. 
Using the representation in Lemma~\ref{l.cascade_rep}, we can compute $\nabla\Phi(t,x) = \E \la \sigma\ra_{\mu,\pi_t,x,-L(t)}$, which implies the particular case in Part~\eqref{i.sup|d_xPhi|<infty}.

We consider the presentation of $\Phi$ in Lemma~\ref{l.cascade_rep} in terms of $\mathbf{f}_\mu$ from~\eqref{e.f(q),<>_q=}. Then,~\eqref{e.d_xPhi=E<F>} is of the form considered in Lemma~\ref{l.interpolation} and Corollary~\ref{c.interpolation}.
Applying Corollary~\ref{c.interpolation}, we have that there is a constant $C_\bi$ such that, for every $(t,x),(t',x')\in[0,\T]\times \R^\D$,
\begin{align*}
    \Ll|\partial_x^\bi\Phi(t,x) - \partial_x^\bi\Phi(t',x')\Rr|\leq C_\bi \Ll(\Ll|L(t)-L(t')\Rr| + \int_0^1\Ll|\pi_t(s)-\pi_{t'}(s)\Rr|\d s+|x-x'|\Rr)
\end{align*}
where $\pi_t$ and $\pi_{t'}$ are given as in Lemma~\ref{l.cascade_rep}. From this, we can verify Part~\eqref{i.d_xPhi_exists}.

We consider the representations for $\Phi_{\mu,L,\alpha}$ and $\Phi_{\mu,L',\alpha'}$ in Lemma~\ref{l.cascade_rep} (with $\pi'_t$ defined similarly for the latter).
Then, we can represent their derivatives in the same way as in~\eqref{e.d_xPhi=E<F>}. We apply Corollary~\ref{c.interpolation} to get a constant $C'_\bi$ such that, for every $(t,x)\in[0,\T]\times \R^\D$,
\begin{align*}
    \Ll|\partial_x^\bi\Phi_{\mu,L,\alpha}(t,x) - \partial_x^\bi\Phi_{\mu,L',\alpha'}(t,x)\Rr| \leq C'_\bi \Ll(\Ll|L(t)-L'(t)\Rr| + \int_0^1\Ll|\pi_t(s)-\pi'_{t}(s)\Rr|\d s\Rr)
    \\
    \leq C'_\bi\Ll(3 \Ll\|L_n-L\Rr\|_{L^\infty} + \Ll\|L\Rr\|_\mathrm{Lip}\Ll\|\alpha-\alpha'\Rr\|_{L^1}  \Rr).
\end{align*}
where the last inequality can deduced in the same way as in~\eqref{e.|Phi_L-Phi_L'|<}. This implies Part~\eqref{i.dxPhi_cvg}.

Lastly, we work under the assumptions in Part~\eqref{i.sup|d_xPhi|<infty} and take an approximation sequence $((L_n,\alpha_n))_{n\in\N}$ of $(L,\alpha)$ consisting of smooth $L_n$ and discrete $\alpha_n$. 
In particular, we can take $L_n$ as in~\eqref{e.L_n=} to ensure that $\Ll(\dot L_n\Rr)_{n\in\N}$ converges to $\dot L$ almost everywhere on $[0,\T]$.
By Lemma~\ref{l.PDE}, $\Phi_n$ satisfies the equation except possibly at finitely many $t$.
Fix any $t\in[0,\T]$, we can integrate the equation for $\Phi_n$ over $[t,\T]$ and then pass to the limit as $n\to\infty$. Parts~\eqref{i.dxPhi_cvg} and~\eqref{i.sup|d_xPhi|<infty} together with the bounded convergence theorem imply, for every $x\in\R^\D$,
\begin{align*}
    \phi(x)-\Phi(t,x)=-\int_t^\T\la \dot L(r)\,,\,\nabla^2\Phi(r,x)+\alpha(r)\nabla \Phi\nabla \Phi)^\intercal(r,x)\ra_{\S^D}\d r.
\end{align*}
Therefore, $\partial_t \Phi$ exists almost everywhere and $\Phi$ satisfies the equation almost everywhere. Under the additional assumption that $L$ is continuously differentiable, using the continuity in Part~\eqref{i.d_xPhi_exists}, we can see that the integrand on the right is continuous at every $r \in[0,T]\setminus\mathsf{disc}_\alpha$. Then, we can obtain the additional claim in Part~\eqref{i.dsdxPhi_exists} for $\partial_t \Phi$ and $\Phi$.
This proves Part~\eqref{i.dsdxPhi_exists} in the case $|\bi|=0$.
For the general case, we can first apply the derivative operator $\partial^\bi_x$ on the equation satisfied by $\Phi_n$ and then repeat the above argument to conclude.
\end{proof}

\begin{lemma}[Time derivative of the solution]\label{l.time_der_sol}
Let $(L,\alpha)$ be a decomposition defined on $[0,\T]$ satisfying that $L$ is smooth and $\alpha$ is continuous and strictly increasing. Let $\Phi=\Phi_{\mu,L,\alpha}$ be the associated Parisi PDE solution. Then, $\Phi$ is differentiable in $t$ at every $t\in[0,\T]\times \R^\D$ with
\begin{align*}
	\partial_t \Phi(t,x) = - \dot L(t)\cdot \Ll(\E \la \sigma\sigma^\intercal\ra_{\mu,\pi_t,x,-L(t)} - \E \la \sigma\sigma'^\intercal \one_{\brho\wedge \brho' > \alpha(t)}\ra_{\mu,\pi_t,x,-L(t)}\Rr)
\end{align*}
where $\la\cdot\ra_{\mu,\pi_t,x,-L(t)}$ is given as in~\eqref{e.f(q),<>_q=} with $\pi_t$ from Lemma~\ref{l.cascade_rep}.
\end{lemma}

\begin{proof}
We only compute the one-sided derivative in $t$ from the right; this is sufficient by Part~\eqref{i.dsdxPhi_exists} of Proposition~\ref{p.reg_Phi}.  Fix any $(t,x)\in[0,\T)\times \R^\D$. We set $s=\alpha(t)$ and $s_\eps = \alpha(t+\eps)$ for every $\eps\in(0,\T-t)$.
Since $\alpha^{-1}$ is strictly increasing, we also have $t=\alpha^{-1}(s)$ and $t+\eps = \alpha^{-1}(s_\eps)$.
We also set $\pi^\eps_t = \pi_t \one_{[s_\eps,1]}$. 

Lemma~\ref{l.cascade_rep} gives that $\Phi(t+\eps,x) = \mathbf{f}_\mu(\pi_{t+\eps},x,-L(t+\eps))$ with $\mathbf{f}_\mu$ given as in~\eqref{e.f(q),<>_q=}. To compute $\partial_t \Phi(t,x)$, we evaluate the limit of $\frac{1}{\eps}\Ll(\mathbf{f}_\mu(\pi_{t+\eps},x,-L(t+\eps)) - \mathbf{f}_\mu(\pi_t,x,-L(t))\Rr)$, which we split into two parts $\mathbf{f}_\mu(\pi_{t+\eps},x,-L(t+\eps))-\mathbf{f}_\mu(\pi^\eps_t,x,-L(t))$ and $\mathbf{f}_\mu(\pi^\eps_t,x,-L(t))-\mathbf{f}_\mu(\pi_t,x,-L(t))$.

We start by treating the second part.
Notice that $\pi^\eps_t$ differs from $\pi_t$ only on the interval $[s,s_\eps]$, on which the difference is bounded by $\|L\|_\mathrm{Lip}\Ll(\alpha^{-1}(s_\eps)-\alpha^{-1}(s)\Rr) = \|L\|_\mathrm{Lip}(t+\eps-t) $. Hence, $\|\pi^\eps_t - \pi_t\|_{L^1}\leq (s_\eps-s)\|L\|_\mathrm{Lip}\eps$. These along with Corollary~\ref{c.interpolation} implies
\begin{align*}\frac{1}{\eps}\Ll|\mathbf{f}_\mu(\pi^\eps_t,x,-L(t)) - \mathbf{f}_\mu(\pi_t,x,-L(t))\Rr| \leq \frac{1}{\eps} \|\pi^\eps_t - \pi_t\|_{L^1}
	\leq (s_\eps-s)\|L\|_\mathrm{Lip}
\end{align*}
which vanishes as $\eps\to0$.

Then, we turn to the first part. 
Comparing $\pi_{t+\eps}$ as in Lemma~\ref{l.cascade_rep} with $\pi^\eps_t$, we can see $\pi_{t+\eps} - \pi^\eps_t = -(L(t+\eps)-L(t))\one_{(s_\eps,1]}$. 
Using this and Lemma~\ref{l.interpolation}, we get
\begin{align*}
    &\frac{1}{\eps}\Ll(\mathbf{f}_\mu(\pi_{t+\eps},x,-L(t+\eps)) - \mathbf{f}_\mu(\pi^\eps_t,x,-L(t))\Rr) 
    \\
    &=-\frac{1}{\eps}(L(t+\eps)-L(t))\cdot \int_0^1 \E \la \sigma\sigma^\intercal - \sigma\sigma'^\intercal \one_{\brho\wedge\brho'> s_\eps} \ra_{\mu,\,\lambda \pi_{t+\eps} + (1-\lambda)\pi^\eps_t,\,x,\, -\lambda L(t+\eps)-(1-\lambda)L(t)} \d \lambda.
\end{align*}
The integrand converges to $\E \la \sigma\sigma^\intercal - \sigma\sigma'^\intercal \one_{\brho\wedge\brho'> s} \ra_{\mu,\pi_t,x,-L(t)}$ as $\eps\to0$ (to see this, one can use Corollary~\ref{c.interpolation}). This along with that the first part vanishes as argued above yields the desired result.
\end{proof}

We record an interesting observation below, even though we will not need it.

\begin{corollary}
Under the same setup of Lemma~\ref{l.time_der_sol}, it holds for every $(t,x)\in[0,\T]\times \R^\D$ that
\begin{align*}
	\dot L(t)\cdot \E \la \sigma\sigma'^\intercal \one_{\brho\wedge\brho'\leq \alpha(t)}\ra_{\mu,\pi_t,x,-L(t)} = \dot L(t)\cdot \E\la \sigma\ra_{\mu,\pi_t,x,-L(t)} \E \la \sigma\ra^\intercal_{\mu,\pi_t,x,-L(t)} \alpha(t).
\end{align*}
\end{corollary}

\begin{proof}
For brevity, we write $\la\cdot\ra=\la\cdot\ra_{\mu,\pi_t,x,-L(t)}$.
By Proposition~\ref{p.reg_Phi}~\eqref{i.dsdxPhi_exists}, $\Phi$ satisfies~\eqref{e.Parisi_PDE} everywhere. Since we can easily compute using the representation in Lemma~\ref{l.cascade_rep} that
\begin{align*}
	\nabla\Phi(t,x) = \E \la \sigma\ra,\qquad \nabla^2 \Phi(t,x) = \E \la \sigma\sigma^\intercal - \sigma\sigma'^\intercal\ra,
\end{align*}
we can rewrite~\eqref{e.Parisi_PDE} into
\begin{align*}
	\partial_t \Phi(t,x) = - \dot L(t) \cdot \Ll(\E \la \sigma\sigma^\intercal - \sigma\sigma'^\intercal\ra + \alpha(t)\E \la \sigma\ra\E \la \sigma\ra^\intercal\Rr)
\end{align*}
Comparing this with the expression in Lemma~\ref{l.time_der_sol}, we get the result.
\end{proof}

\section{Proof of the main result}
\label{s.proof}

Recall that $P_1$ is the distribution of a single spin.
We use the following terminology in this section.
Given a decomposition $(L,\alpha)$ of some $q\in\mcl Q_\infty$ on $[0,\T]$, we say that $\Phi$ is \textbf{associated with $(L,\alpha)$} if $\Phi = \Phi_{P_1,L,\alpha}$ is the Parisi PDE solution given in Definition~\ref{d.sol_pde} with $\mu=P_1$.
We also note that 
\begin{align}\label{e.q(1)=aq(1)=Lalpha^-1(1)}
    q(1)=\aq(1)=L\circ\alpha^{-1}(1)
\end{align}
due to the definitions of $q(1)$ in~\eqref{e.q(1)=} and $\aq$ in~\eqref{e.aq=}.

We also need a stochastic process, the sample paths of which can be interpreted as characteristics of the Parisi PDE.
Fix a probability space on which there is a standard $\D$-dimensional Wiener process $W=(W_s)_{s\in [0,\T]}$, where $[0,\T]$ is the common domain for $L$ and $\alpha$.
Given the solution $\Phi$ associated with $(L,\alpha)$, we say that the process $X=(X_t)_{t\in[0,\T]}$ is \textbf{associated with $(L,\alpha)$} if $X$ is the strong solution (see \cite[Definition~2.1 in Chapter~5]{karatzas1991brownian}) of the following SDE
\begin{align}\label{e.X_t=}
\begin{split}
    \d X_t &= 2\alpha(t)\dot L(t) \nabla\Phi(t,X_t)\d t + \Ll(2\dot L\Rr)^\frac{1}{2}(t)\d W_t,
    \\
    X_0 & = x
\end{split}
\end{align}
for some $x\in\R^\D$ to be specified in different contexts. The existence and uniqueness of the strong solution follows from the regularity of $\nabla\Phi$ in Proposition~\ref{p.reg_Phi}, that $\dot L$ is bounded a.e., and the standard results~\cite[Theorems~2.5 and 2.9 in Chapter~5]{karatzas1991brownian}.

\subsection{Computation along characteristics}

Later, we need some approximation using more regular $(L,\alpha)$. For this, we need the following result.

\begin{lemma}[Continuity of $X$ in $(L,\alpha)$]\label{l.X_n_cvg}
For each $n$, let $\Phi_n$ and $X^n$ be associated with some decomposition $(L_n,\alpha_n)$ defined on $[0,\T]$ and let $X^n_0= x_n\in\R^\D$.
Assume that $(x_n)_{n\in\N}$ converges to $x\in\R^\D$, $(\alpha_n)_{n\in\N}$ converges to $\alpha$ in $\cM$, $\sup_{n\in\N}\Ll\|\dot L_n\Rr\|_{L^\infty}<\infty$, and that $\Ll(\dot L_n\Rr)_{n\in\N}$ converges to $L$ pointwise a.e.\ on $[0,\T]$. Let $\Phi$ and $X$ be associated with $(L,\alpha)$ and let $X_0=x$. Then, $\lim_{n\to\infty}\E \Ll[\sup_{t\in[0,\T]}\Ll|X^n_t-X_t\Rr|^p\Rr] =0$ for every $p\geq 1$.
\end{lemma}

\begin{proof}
For simplicity, we write $\nabla\bar\Phi = \nabla \Phi(\cdot,X_\cdot)$ and $\nabla\bar\Phi_n = \nabla \Phi_n(\cdot,X^n_\cdot)$.
Fix any $t\in[0,\T]$, we set
\begin{align*}
    \mathbf{I}= \int_0^t 2\alpha_n\dot L_n\nabla\bar \Phi_n\d r  -  \int_0^t 2\alpha\dot L \nabla\bar \Phi_n\d r,\qquad \mathbf{II} = \int_0^t\Ll(2\dot L_n\Rr)^\frac{1}{2}\d W - \int_0^t\Ll(2\dot L\Rr)^\frac{1}{2}\d W
\end{align*}
and thus $X^n_t - X_t = x_n-x+\mathbf{I} + \mathbf{II}$. We start by splitting
\begin{align*}
    \mathbf{I} = \int_0^t 2(\alpha_n-\alpha)\dot L_n \nabla\bar \Phi_n\d r +  \int_0^t 2\alpha_n\Ll(\dot L_n - \dot L\Rr)\nabla\bar \Phi_n\d r  + \int_0^t 2\alpha_n\dot L_n\Ll(\nabla\bar \Phi_n - \nabla\bar \Phi \Rr)\d r.
\end{align*}
For the last term, we can further split
\begin{align*}
    \nabla\Phi_n(r) - \nabla\bar \Phi(r) = \Ll(\nabla\Phi_n(r,X^n_r) - \nabla\Phi(r,X^n_r) \Rr)+\Ll(\nabla\Phi(r,X^n_r)- \nabla\Phi(r,X_r)\Rr).
\end{align*}
Due to the uniform bound on $\dot L_n$ and the bounds in Proposition~\ref{p.reg_Phi}~\eqref{i.sup|d_xPhi|<infty}, there is a constant $C>0$ such that
\begin{align*}
    C^{-1}|\mathbf{I}| \leq \Ll\|\alpha_n-\alpha\Rr\|_{L^1} + \Ll\|\dot L_n-\dot L\Rr\|_{L^1} +  \Ll\|\nabla \Phi_n - \nabla \Phi\Rr\|_{L^\infty} + \|\nabla\Phi\|_\mathrm{Lip}\int_0^t\Ll|X^n_r-X_r\Rr|\d r.
\end{align*}
Henceforth, we allow the deterministic constant $C$ to change from instance to instance but independent of $n$ or $t$. By the assumption on various convergences and Proposition~\ref{p.reg_Phi}~\eqref{i.dxPhi_cvg}, we can see that the first three terms on the right-hand side vanish as $n\to\infty$. Hence, we get
\begin{align}\label{e.|X^n_t-X_t|<}
    \Ll|X^n_t-X_t\Rr| \leq o_n(1)+C\int_0^t\Ll|X^n_r- X_r\Rr|\d r + |\mathbf{II}|. 
\end{align}
where $\lim_{n\to\infty} o_n(1)=0$ and $o_n(1)$ is dependent of $n$ or $t$.

Fix any $p\geq 1$.
By the BDG inequality~\cite[Theorem~3.28 in Chapter~3]{karatzas1991brownian}, we have
\begin{align*}
    \E \sup_{t\in[0,\T]}\Ll|\mathbf{II}\Rr|^p \leq C \Ll(\int_0^\T \Ll|\Ll(2\dot L_n\Rr)^\frac{1}{2} - \Ll(2\dot L\Rr)^\frac{1}{2}\Rr|^2 \d r\Rr)^{p/2}.
\end{align*}
(A simpler argument is also possible using the observation that $\mathbf{II}$ is a deterministic time change of Brownian motion.)
By the assumption on $\dot L_n$ and the fact that the matrix square root is a continuous function, we can see that the right-hand side vanishes as $n\to\infty$. We can thus absorb $\E |\mathbf{II}|^p$ into $o_n(1)$ in~\eqref{e.|X^n_t-X_t|<} and get
\begin{align*}
    a_n(t) \leq o_n(1)+C\int_0^ta_n(r)\d r,\quad\forall t\in[0,\T].
\end{align*}
where $a_n(r) = \E\Ll[\sup_{r'\in[0,r]} \Ll|X^n_{r'}-X_{r'}\Rr|^p\Rr]$. This along with Gronwall's inequality implies $\lim_{n\to\infty} a_n(\T) = 0$, which is the desired result.
\end{proof}

As another step of preparation, we need the following, where the setting is slightly more general ($\Phi$ is not required to have terminal condition~\eqref{e.phi=initial} and $X_0$ is not specified).

\begin{lemma}[Martingale along the linearization of the Parisi PDE]\label{l.mtgl}
Let $\Phi$ be a classical solution of~\eqref{e.Parisi_PDE} with continuously differentiable $L$, continuous $\alpha$, and a smooth terminal condition.
Let $X$ be a strong solution of~\eqref{e.X_t=}.
Let $g:[0,\T]\times\R^\D\to\R$ be continuous and let $\Psi$ be a classical solution of
\begin{align}\label{e.Psi_g_eqn}
    \partial_t\Psi(t,x) + \la \dot L(t),\, \nabla^2 \Psi(t,x) + 2\alpha(t) \nabla\Psi\nabla\Phi^\intercal(t,x) \ra_{\S^\D} + g(t,x)=0
\end{align}
for $(t,x)\in[0,\T]\times\R^\D$. Then,
\begin{align*}\Psi(t,X_t)+\int_0^t g(r,X_r)\d r 
\end{align*}
is a martingale with index $t\in[0,\T]$.
\end{lemma}

\begin{proof}
Notice that the quadratic variation of $X$ is given by $\la X\ra_t= 2\dot L(t)$. Using this and Ito's formula, we get
\begin{align*}
    \d \Psi(t,X_t) = \Ll(\partial_t\Psi(t,X_t)+\la \dot L(t),\nabla^2 \Psi(t,X_t) \ra_{\S^\D}\Rr)\d t + \la \nabla \Psi(t,X_t), \d X_t\ra_{\R^\D}.
\end{align*}
Inserting the expression of $\d X_t$ in~\eqref{e.X_t=} and using~\eqref{e.Psi_g_eqn}, we get
\begin{align*}
    \d \Psi(t,X_t) = - g(t,X_t)\d t + \la \nabla\Psi(t,X_t),\Ll(2\dot L\Rr)^\frac{1}{2}(t)\d W_t\ra_{\R^\D}
\end{align*}
which implies the desired result.
\end{proof}

We can get the following result using the It\^o calculus and the fact that $\Phi$ solves the Parisi PDE.

\begin{lemma}[Relation between matrix-valued processes]\label{l.R,A_rel}
Let $\Phi$ and $X$ be associated with some decomposition $(L,\alpha)$.
For $t\in [0,\T]$, define
\begin{gather}\label{e.R(t)=A_t=}
    R(t) = \E \Ll[\nabla\Phi\nabla\Phi^\intercal(t,X_t)\Rr];
    \qquad
    A_t= \nabla^2\Phi(t,X_t). 
\end{gather}
Then, the following holds for every $t\in[0,\T]$:
\begin{gather}
    R(t) = R(0)+ 2 \int_0^t \E \Ll[A_r^\intercal \dot L(r) A_r\Rr]\d r; \label{e.R(t)=R(0)+2int}
    \\
    \E \Ll[A_t\Rr] = \E \Ll[A_\T\Rr] + \int_t^\T \alpha(r)\dot R(r)\d r.\label{e.E[nabla^2Phi]=}
\end{gather}
\end{lemma}

We emphasize that both $R(\cdot)$ and $A_\cdot$ depend on the initial data $X_0$ which is kept implicit here.

\begin{proof}
For $m\in\N$ and $i_1,\dots,i_m\in\{1,\dots,\D\}$, we write $\Phi_{i_1i_2\dots i_m} = \partial_{x_{i_1}} \partial_{x_{i_2}}\cdots \partial_{x_{i_m}}\Phi$.
We recall the following identities from \cite[Lemma~3.3]{chen2023parisi} (with $\mu$ therein replaced by $2L$) obtained by using the It\^o formula and the Parisi PDE~\eqref{e.Parisi_PDE}:
\begin{align}
    \d \Ll(\nabla\Phi(t,X_t) \Rr) & =  \nabla^2\Phi(t,X_t) \Ll(2\dot L\Rr)^\frac{1}{2}(t)\d W_t; \label{e.nablaPhi(b,X(b))}
    \\
    \d \Ll(\Phi_{kl}(t,X_t)\Rr)  & = \sum_{i,j=1}^\D\Ll(-  \alpha (t)\Ll(2\dot L\Rr)_{ij}(t)\Phi_{ik}\Phi_{jl}(t,X_t)\d t +\Phi_{ikl}(t,X_t)\Ll(2\dot L\Rr)^\frac{1}{2}_{ij}(t)\d W_{j,t}  \Rr)\label{e.nabla^2Phi(b,X(b))},
\end{align}
where $k,l\in\{1,\dots,\D\}$.
Then, we compute
\begin{align*}
    \d \Phi_k \Phi_l(t,X_t) = \Phi_k(t,X_t)\d \Phi_l(t,X_t) +  \Phi_l(t,X_t)\d \Phi_k(t,X_t) + \d \la\Phi_k(\cdot,X_\cdot),\, \Phi_l(\cdot,X_\cdot)\ra_t
    \\
    \stackrel{\eqref{e.nablaPhi(b,X(b))}}{=} \text{martingale} + \sum_{i,j=1}^\D \Ll(2\dot L\Rr)_{ij}(t)\Phi_{ik}\Phi_{jl}(t,X_t)\d t
\end{align*}
where $\la\Phi_k(\cdot,X_\cdot),\, \Phi_l(\cdot,X_\cdot)\ra_t$ denotes the corresponding quadratic variation.
The above display implies
\begin{align}\label{e.R_kl(t)=}
    R_{kl}(t) = R_{kl}(0)+ \int_0^t \sum_{i,j=1}^\D \E \Ll[\Ll(2\dot L\Rr)_{ij}(r)\Phi_{ik}\Phi_{jl}(r,X_r)\Rr] \d r
\end{align}
which gives~\eqref{e.R(t)=R(0)+2int}.
We also have
\begin{align*}
    \E \Phi_{kl}(t,X_t) &\stackrel{\eqref{e.nabla^2Phi(b,X(b))}}{=} \E \Phi_{kl}(\T,X_\T)  + \E \int_t^\T \alpha(s)\sum_{i,j=1}^\D \Ll(2\dot L\Rr)_{ij}(r)\Phi_{ik}\Phi_{jl}(r,X_r)\d r
    \\
    &\stackrel{\eqref{e.R_kl(t)=}}{=} \E \Phi_{kl}(\T,X_\T) + \E \int_t^\T \alpha(r) \dot R_{kl}(r) \d r,
\end{align*}
which yields~\eqref{e.E[nabla^2Phi]=}.
\end{proof}

We recall that the function $\psi$ is as in \eqref{e.psi=}. 
\begin{lemma}\label{l.R(0),A_0}
Let $\Phi$ and $X$ be associated with some decomposition $(L,\alpha)$ of some $q\in\mcl Q_\infty$. 
Let $X(0)= \sqrt{2L(0)}\eta$ for a standard $\R^\D$-valued Gaussian vector $\eta$ independent of everything else.
Let $R(\cdot)$ and $A_\cdot$ be given as in~\eqref{e.R(t)=A_t=}. We have
\begin{gather}\label{e.EPhi=-psi}
    \E_\eta \Phi\Ll(0, \sqrt{2L(0)}\eta\Rr) = - \psi(q) ,\qquad \E_\eta  A_0  = \E \Ll[\la \sigma\sigma^\intercal \ra_q  -\la \sigma\sigma'^\intercal\ra_q\Rr]
\end{gather}
where $\E_\eta$ averages only over the randomness of $\eta$ and the Gibbs measure is given by
\begin{align}\label{e.<>_q_correct}
    \la \cdot \ra_q \propto \exp\Ll(\sqrt{2}\sigma\cdot w^q(\brho) - q(1)\cdot\sigma\sigma^\intercal\Rr)\d P_1(\sigma)\d\mathfrak{R}(\brho).
\end{align}
If $L(0)=0$, we also have $R(0) = \E \la \sigma\ra_q \Ll(\la\sigma\ra_q\Rr)^\intercal$.
\end{lemma}

\begin{proof}Using the representation of $\Phi$ in Lemma~\ref{l.cascade_rep} at $t=0$ with $\mu=P_1$ and~\eqref{e.q(1)=aq(1)=Lalpha^-1(1)}, we have that $\Phi\Ll(0,x+\sqrt{2L(0)}\eta\Rr)$ is equal to
\begin{align*}
    \E \iint \exp\Ll(\sqrt{2}\sigma\cdot w^{\aq -L(0)}(\brho)-q(1)\cdot \sigma\sigma^\intercal +\Ll(x+\sqrt{2L(0)}\eta\Rr)\cdot \sigma\Rr)\d P_1(\sigma)\d \mathfrak{R}(\brho)
\end{align*}
where $\E$ averages over all randomness except for that of $\eta$. In view of the covariance formula~\eqref{e.E[ww]=}, we can see that, conditioned on $\mathfrak{R}$, $w^{\aq-L(0)}+ \sqrt{L(0)}\eta$ has the same distribution as $w^{\aq}$, and we can substitute $w^{\aq}(\brho)$ for $w^{q}(\brho)$ as explained in Remark~\ref{r.left-cts}. Setting $x=0$ and taking $\E_\eta$, we can get the first relation in~\eqref{e.EPhi=-psi} by comparing the above with $\psi$ given in~\eqref{e.psi=}.

Since $A_0 = \nabla^2\Phi \Ll(0,\sqrt{2L(0)}\eta\Rr)$, we can compute $A_0  = \E \Ll[\la \sigma\sigma^\intercal \ra  -\la \sigma \sigma'^\intercal\ra \Rr]$ where 
$\la\cdot\ra$ is the random Gibbs measure proportional to
\begin{align*}
    \exp\Ll(\sqrt{2}\sigma\cdot w^{\aq -L(0)}(\brho)-q(1)\cdot \sigma\sigma^\intercal +\sqrt{2L(0)}\eta\cdot \sigma\Rr)\d P_1(\sigma)\d \mathfrak{R}(\brho).
\end{align*}
By the same reason as above, we have $ \E_\eta\E \la \cdot\ra = \E \la \cdot\ra_q$. Hence, we can verify the relation involving $A_0$ in~\eqref{e.EPhi=-psi}.

Since $R(0) = \nabla\Phi\nabla\Phi^\intercal  \Ll(0,\sqrt{2L(0)}\eta\Rr)$, we can compute $R(0)  = \E \la \sigma \ra\Ll(\E\la \sigma\ra\Rr)^\intercal$ for the same $\la\cdot\ra$ as above. If $L(0)=0$, we simply have $\la\cdot\ra=\la\cdot\ra_q$ which gives the expression for $R(0)$.
\end{proof}

\begin{lemma}[Representation of a functional derivative]\label{l.d_eps_Phi^eps}
Let $(L,\alpha)$ be any composition of some $q\in\mcl Q_\infty$ defined on $[0,\T]$. Let $L':[0,\T]\to \S^\D_+$ be any smooth increasing path. For each $\eps\in[0,1)$, let $\Phi^\eps$ be associated with $(L+\eps L',\alpha)$. Let $\Phi=\Phi^0$ and let $X$ be the process associated with $(L,\alpha)$ with $X_0 = x_0\in\R^\D$. We have
\begin{align}\label{e.d/depsPhi(0,x_0)=-intLRalpha}
    \frac{\d}{\d\eps}\Big|_{\eps=0} \Phi^\eps(0,x_0) & = -L'(0)\cdot \Ll(A_0 + \alpha(0)R(0)\Rr)-\int_0^\T L'(t)\cdot R(t)\d\alpha(t) 
\end{align}
where $R(\cdot)$ and $A_\cdot$ are given in~\eqref{e.R(t)=A_t=}.
\end{lemma}

We view~\eqref{e.d/depsPhi(0,x_0)=-intLRalpha} as a generalization of \cite[Lemma~3.7]{talagrand2006parisi} and~\cite[(19)]{auffinger2015properties} that focused on the case $\D=1$.

\begin{proof}We proceed in five steps. In the first two steps, we assume that $L$ is
continuously differentiable and that $\alpha$ is continuous and strictly
increasing. Under these assumptions, we verify~\eqref{e.d/dPhi^eps=}.
In the remaining three steps, we use approximation arguments to extend
\eqref{e.d/dPhi^eps=} to the general case. Whenever the derivative exists, we
write $\partial_\eps = \frac{\mathrm{d}}{\mathrm{d}\eps}\big|_{\eps=0}$ and
$\Psi = \partial_\eps \Phi^\eps$, with the latter identity understood
pointwise.

\textit{Step~1.}
Under the additional assumption on $(L,\alpha)$, we show
\begin{align}\label{e.d/dPhi^eps=}
    \Psi(0,x_0)= -L'(\T)\cdot N(\T) + \int_0^\T \dot L'(t) \cdot N(t)\d t,
\end{align}
where
\begin{align*}N(t) = \E \Ll[\nabla^2\Phi(t,X_t)+\alpha(t)\nabla\Phi\nabla\Phi^\intercal(t,X_t)\Rr],\quad\forall t\in [0,\T].
\end{align*}
First, we argue that $\Psi$, $\partial_t \Psi$, and $\partial^\bi_x\Psi$ are well defined and that the derivatives commute: $\partial_t \Psi = \partial_\eps \partial_t \Phi^\eps$ and $\partial^\bi_x \Psi = \partial_\eps \partial^\bi_x \Phi^\eps$. 
For the derivative in $t$, we recall that $\partial_t \Phi^\eps$ has the expression given in Lemma~\ref{l.time_der_sol}, which has the form appearing on the left-hand side of the interpolation computation~\eqref{e.E<F>_q-E<F>_q'}.
Based on this, we use Lemma~\ref{l.interpolation} to compute $\partial_\eps\partial_t \Phi^\eps$. On the other hand, we can first compute $\partial_\eps \Phi^\eps$ using Lemma~\ref{l.cascade_rep} and Lemma~\ref{l.interpolation}. Then, we can compute $\partial_t\partial_\eps \Phi^\eps$ following the same argument in Lemma~\ref{l.time_der_sol}. Comparing these expressions, we can verify $\partial_\eps \partial_t \Phi^\eps = \partial_t\partial_\eps \Phi^\eps$. The main point is that all objects are computable because they are expectations of bounded functions of the spin and cascade variables. The detail is tedious and omitted here. The same can be said about derivatives in $x$.

Allowed by this, we can differentiate the equation~\eqref{e.Parisi_PDE} satisfied by $\Phi^\eps$ in $\eps$ to see that $\Psi$ is the classical solution of~\eqref{e.Psi_g_eqn} with $g$ given by 
\begin{align*}
    g(t,x) = \la \dot L'(t),\, \nabla^2 \Phi(t,x) +\alpha(t) \nabla\Phi\nabla\Phi^\intercal (t,x)\ra_{\S^\D}.
\end{align*}
By the martingale property proved in Lemma~\ref{l.mtgl}, we thus have
\begin{align}\label{e.Psi(0,x_0)=mtgl}
    \Psi(0,x_0) = \E \Ll[\Psi(\T,X_\T) + \int_0^\T g(t,X_t)\d s\Rr] .
\end{align}
Now, let us evaluate $\E\Ll[\Psi(\T,X_\T)\Rr]$. In view of the terminal condition given in~\eqref{e.phi=initial} with $\mu=P_1$, we can compute for every $x\in \R^\D$
\begin{align*}
	\Psi(\T,x) = - q'(1)\cdot \la\sigma\sigma^\intercal \ra,\qquad \nabla\Phi(\T,x) = \la \sigma\ra,\qquad \nabla^2\Phi(\T,x) = \la \sigma\sigma^\intercal\ra - \la\sigma\ra\la\sigma\ra^\intercal
\end{align*}
where $q'=L'\circ\alpha^{-1}$ and $\la\cdot\ra$ is the deterministic Gibbs measure associated with the right-hand side in~\eqref{e.phi=initial} at $x$. 
Also, since we have assumed that $\alpha$ is strictly increasing in the first two steps, we have $\alpha(\T)=1$ and $\alpha^{-1}(1)=\T$. In particular, we have $q'(1)=L'\circ\alpha^{-1}(1)=L'(\T)$.
Using these, we get
\begin{align*}
	\Psi(\T,x) = -L'(\T)\cdot \Ll(\nabla^2\Phi(\T,x) + \alpha(\T)\nabla\Phi\nabla\Phi^\intercal(\T,x)\Rr),\quad\forall x\in \R^\D.
\end{align*}
This along with~\eqref{e.Psi(0,x_0)=mtgl} yields~\eqref{e.d/dPhi^eps=} under the assumption that $L$ is continuously differentiable and $\alpha$ is continuous and strictly increasing.

\textit{Step~2.}
Continuing with the regularity assumptions on $(L,\alpha)$, we now verify~\eqref{e.d/depsPhi(0,x_0)=-intLRalpha}.
We want to use the relations from Lemma~\ref{l.R,A_rel}. Recall $R(\cdot)$ and $A_\cdot$ given in~\eqref{e.R(t)=A_t=}. For brevity, we write $\mathbf{a}(\cdot) = \E [A_\cdot]$. Then, we have
\begin{align*}N(t) = \mathbf{a}(t) + \alpha(t)R(t),\quad\forall t\in[0,\T].
\end{align*}
Inserting this into~\eqref{e.d/dPhi^eps=}, we get
\begin{align*}
    \Psi(0,x_0) = - L'(\T)\cdot \Ll(\mathbf{a}(\T) + \alpha(\T)R(\T)\Rr)+ \int_0^\T \dot L'(t)\cdot \mathbf{a}(t) \d t + \int_0^\T \dot L'(t)\cdot \alpha(t)R(t) \d t .
\end{align*}
Next, we compute the second term on the right using the integration by parts (IBP) and results from Lemma~\ref{l.R,A_rel}:
\begin{align*}
    \int_0^\T \dot L'(t)\cdot \mathbf{a}(t) \d t &\stackrel{\text{(IBP)}}{=} L'(\T)\cdot \mathbf{a}(\T)-L'(0)\cdot \mathbf{a}(0) -\int_0^\T L'(t)\cdot\dot{\mathbf{a}}(t)\d t
    \\
    &\stackrel{\eqref{e.E[nabla^2Phi]=}}{=} L'(\T)\cdot \mathbf{a}(\T)-L'(0)\cdot \mathbf{a}(0) + \int_0^\T L'(t)\cdot \alpha(t)\dot R(t)\d t
    \\
    &\stackrel{\text{(IBP)}}{=} L'(\T)\cdot \Ll(\mathbf{a}(\T)+\alpha(\T)R(\T)\Rr) - L'(0)\cdot \Ll(\mathbf{a}(0)+\alpha(0)R(0)\Rr)
    \\
    &\qquad  - \int_0^\T\dot L'(t)\cdot \alpha(t)R(t)\d t - \int_0^\T L'(t)\cdot R(t) \d \alpha(t).
\end{align*}
Plugging this back to the previous display, we get~\eqref{e.d/depsPhi(0,x_0)=-intLRalpha}. Notice that $A_0$ is deterministic (since $X_0=x_0$ is deterministic) and thus $\mathbf{a}(0)= A_0$.

\textit{Step~3.}
We want to conclude by an approximation argument. In this step, we introduce the
necessary setup, and we complete the argument in the next two steps.
Fix any $(L,\alpha)$. For each $n\in\N$, let $L_n$ be the mollified version of $L$ given in~\eqref{e.L_n=}. In particular, $\Ll(L_n(\T)\Rr)_{n\in\N}$ converges to $L(\T)$ and $\Ll(\dot L_n\Rr)_{n\in\N}$ converges to $\dot L$ pointwise a.e.\ on $[0,\T]$. Let $(\alpha_n)_{n\in\N}$ be a sequence of continuous and strictly increasing p.d.f.s defined on $[0,\T]$ and converging to $\alpha$ in $\cM$.
In particular, we can choose $(\alpha_n)_{n\in\N}$ in such a way that $(\alpha_n)_{n\in\N}$ converges pointwise a.e.\ to $\alpha$; 
and $(\alpha_n(0))$ converges $\alpha(0)$. Notice that $((L_n,\alpha_n))_{n\in\N}$ satisfies the conditions in Lemma~\ref{l.X_n_cvg}.

For each $n\in\N$ and $\eps\in[0,1)$, let $\Phi^\eps_n$ be associated with $(L_n+\eps L',\alpha_n)$, set $\Phi_n=\Phi^0_n$, and $\Psi_n = \partial_\eps \Phi_n^\eps$. Let $X^n$ be associated with $(L_n,\alpha_n)$.
For each $n\in\N$, we define $R_n(\cdot)$ and $A^n_\cdot$ as in~\eqref{e.R(t)=A_t=} with $\Phi_n,X^n$ substituted for $\Phi,X$ therein. 
In the first two steps, we have shown~\eqref{e.d/depsPhi(0,x_0)=-intLRalpha} for these approximations, namely, 
\begin{align}\label{e.Psi_n(0,x_0)=}
    \Psi_n(0,x_0) = -L'(0)\cdot \Ll(A^n_0 + \alpha_n(0)R_n(0)\Rr) -\int_0^\T L'(t)\cdot R_n(t)\d \alpha_n(t).
\end{align}
In the next two steps, we show that the right-hand side and the left-hand side in this display converge, respectively, to those in~\eqref{e.d/depsPhi(0,x_0)=-intLRalpha}.

\textit{Step~4.}
We show the convergence of the right-hand side of~\eqref{e.Psi_n(0,x_0)=}.
Using the uniform bound on derivatives of $\Phi_n$ and $\Phi$ and the uniform convergence of derivatives of $\Phi_n$ in Proposition~\ref{p.reg_Phi}~\eqref{i.sup|d_xPhi|<infty} and~\eqref{i.dxPhi_cvg}, respectively, we can find a constant $C$ and a vanishing sequence $(o_n(1))_{n\in\N}$ of positive real numbers such that, uniformly in $n$ and $t$,
\begin{align*}
	\Ll|R_n(t)-R(t)\Rr|\leq o_n(1)+C\E\Ll|X^n_t-X_t\Rr|.
\end{align*}
Then, using Lemma~\ref{l.X_n_cvg}, we can see that $(R_n)_{n\in\N}$ converges to $R$ uniformly on $[0,\T]$. Therefore, the right-hand side of~\eqref{e.d/depsPhi(0,x_0)=-intLRalpha} is stable under this approximation:
\begin{align*}
    \lim_{n\to\infty} \int_0^\T L'(t)\cdot R_n(t)\d\alpha_n(t) = \int_0^\T L'(t)\cdot R(t)\d\alpha(t).
\end{align*}
In view of~\eqref{e.R(t)=A_t=}, we have $R_n(0)= \nabla\Phi_n\nabla\Phi^\intercal_n(0,x_0)$ and $A^n_0 = \nabla^2\Phi_n(0,x_0)$. Hence, we also have $\lim_{n\to\infty} R_n(0)=R(0)$ and $\lim_{n\to\infty} A^n_0 = A_0$ by Proposition~\ref{p.reg_Phi}~\eqref{i.dxPhi_cvg}. In conclusion, the right-hand side in~\eqref{e.Psi_n(0,x_0)=} converges to that in~\eqref{e.d/depsPhi(0,x_0)=-intLRalpha}.

\textit{Step~5.}
We show the convergence of the left-hand side of~\eqref{e.Psi_n(0,x_0)=}.
Namely, that $\Psi_n(0,x_0)$ converges to $\Psi(0,x_0)$. 
To compute $\Psi_n(0,x_0) = \partial_\eps \Phi_n^\eps(0,x_0)$, we consider the representation of $\Phi_n^\eps(0,x_0)$ given in Lemma~\ref{l.cascade_rep}:
\begin{align*}
    \Phi_n^\eps(0,x_0)=\E\log\iint\exp\Ll(\sqrt{2}\sigma\cdot w^{L_n^\eps\circ\alpha^{-1}_n}(\brho)-\sigma\cdot x_0-\sigma\sigma^\intercal\cdot L_n^\eps\circ\alpha_n^{-1}(1)\Rr)\d\mu(\sigma)\d\fR(\brho)
\end{align*}
where $L_n^\eps=L_n+\eps L'$. Let us denote the Gibbs measure associated with $\Phi^\eps_n(0,x_0)$ by $\la\cdot\ra_{L^\eps_n\circ\alpha_n^{-1}}$ as it only depends on $L^\eps_n\circ\alpha_n^{-1}$.
Then, we can compute the derivative using the interpolation~\eqref{e.f(q)-f(q')=int} in Lemma~\ref{l.interpolation} to get
\begin{align*}
	\Psi_n(0,x_0) = \E \la \sigma\sigma^\intercal \cdot L'\circ\alpha_n^{-1}(1) - \sigma\sigma'^\intercal \cdot L'\circ\alpha_n^{-1}\Ll(\brho\wedge\brho'\Rr) \ra_{L_n\circ\alpha^{-1}_n}.
\end{align*}
Let us denote the random variable inside $\E\la\cdot\ra_{L_n\circ\alpha^{-1}_n}$ by $f_{\alpha_n}$, highlighting the dependence on $\alpha_n$, so that $\Psi_n(0,x_0)=\E\la f_{\alpha_n}\ra_{L_n\circ\alpha^{-1}_n}$.
Similarly, we have
\begin{align*}
    \Psi(0,x_0) = \E \la \sigma\sigma^\intercal \cdot L'\circ\alpha^{-1}(1) - \sigma\sigma'^\intercal \cdot L'\circ\alpha^{-1}\Ll(\brho\wedge\brho'\Rr) \ra_{L\circ\alpha^{-1}}
\end{align*}
and we denote the random variable inside $\E\la\cdot\ra_{L\circ\alpha^{-1}}$ by $f_{\alpha}$.
Then, we can estimate 
\begin{align}\label{e.|Phi_n-Phi|}
	\Ll|\Psi_n(0,x_0)-\Psi(0,x_0)\Rr| \leq \Ll|\E \la f_{\alpha_n} \ra_{L_n\circ\alpha^{-1}_n} - \E \la f_{\alpha} \ra_{L_n\circ\alpha^{-1}_n} \Rr| + \Ll|\E \la f_{\alpha} \ra_{L_n\circ\alpha^{-1}_n} - \E \la f_{\alpha} \ra_{L\circ\alpha^{-1}} \Rr|.
\end{align}
Let us bound each term on the right.
Using the Lipschitzness of $L'$, there is a constant $C>0$ such that
\begin{align}
    \Ll|\E \la f_{\alpha_n} \ra_{L_n\circ\alpha^{-1}_n} - \E \la f_{\alpha} \ra_{L_n\circ\alpha^{-1}_n} \Rr| \leq C\Ll|\alpha^{-1}_n(1)-\alpha^{-1}(1)\Rr| + C\E \la \Ll|\alpha^{-1}_n-\alpha^{-1}\Rr|(\brho\wedge\brho')\ra_{L_n\circ\alpha^{-1}_n}\notag
    \\
    \stackrel{\text{L.\ref{l.invar_cascade}}}{=} C\Ll|\alpha^{-1}_n(1)-\alpha^{-1}(1)\Rr| + C\Ll\|\alpha^{-1}_n-\alpha^{-1}\Rr\|_{L^1}\label{e.|E<f>-E<f>|_1}
\end{align}
where we used the invariance of cascades to see that $\brho\wedge\brho'$ distributes uniformly under $\E\la\cdot\ra_{L_n\circ\alpha^{-1}_n}$.
Next, we use the interpolation argument as in Corollary~\ref{c.interpolation} (with $L_n\circ\alpha^{-1}_n$ and $L\circ\alpha^{-1}$ substituted for $q$ and $q'$ therein) to find a constant $C$ such that
\begin{align}\label{e.|E<f>-E<f>|_2}
\begin{split}
    &\Ll|\E \la f_\alpha \ra_{L_n\circ\alpha^{-1}_n} - \E \la f_{\alpha} \ra_{L\circ\alpha^{-1}} \Rr|
    \\
    &\leq C\|f_\alpha\|_{L^\infty}\Ll(\Ll|L_n\circ\alpha^{-1}_n(1)-L\circ\alpha^{-1}(1)\Rr|+ \Ll\|L_n\circ\alpha^{-1}_n-L\circ\alpha^{-1}\Rr\|_{L^1}\Rr).
\end{split}
\end{align}
Due to our choice of $(L_n,\alpha_n)$ in Step~3, the terms on the right of~\eqref{e.|E<f>-E<f>|_1} and~\eqref{e.|E<f>-E<f>|_2} vanish as $n\to \infty$. Inserting them to~\eqref{e.|Phi_n-Phi|}, we get $\lim_{n\to\infty}\Psi_n(0,x_0)=\Psi(0,x_0)$.

\textit{Conclusion.}
In Steps~1 and~4, we established~\eqref{e.d/depsPhi(0,x_0)=-intLRalpha} for regular
pairs $(L,\alpha)$. In Step~3, we set up an approximation scheme by selecting a
sequence of regular pairs $(L_n,\alpha_n)$ such that
\eqref{e.Psi_n(0,x_0)=} converges to the corresponding quantity for a general
pair $(L,\alpha)$.  
In Steps~4 and~5, we verified that one may pass to the limit in
\eqref{e.Psi_n(0,x_0)=}, thereby obtaining
\eqref{e.d/depsPhi(0,x_0)=-intLRalpha}.  
This completes the proof.
\end{proof}

We recall the definition of joint decompositions of paths from Subsection~\ref{s.joint_decomp}, as well as the initial condition $\psi$ as in~\eqref{e.psi=} and its differentiability in~\eqref{e.psi_diff}.
In the following, we often assume the following common setting:
\begin{enumerate}[start=1,label={\rm (\textbf{S})}]
    \item \label{i.setting}
    Let $p, q\in\mcl Q_\infty$ satisfy $p=\partial_q \psi(q)$ and let $(L_p,L_q,\alpha)$ be any joint decomposition of $(p,q)$ on $[0,\T]$.
Let $\Phi$ and $X$ be associated with $( L_q, \alpha)$. Let $R(\cdot)$ and $A_\cdot$ be given as in~\eqref{e.R(t)=A_t=} corresponding to $(L_q,\alpha)$.
\end{enumerate}

In addition to~\ref{i.setting}, we often need to specify additional properties of the decomposition and the value of $X_0$.

\begin{lemma}[Representation of $\partial_q\psi(q)$, Part~1]\label{l.rep_grad_psi}
Under~\ref{i.setting}, let $X_0 = \sqrt{2L_q(0)}\eta$ for a standard $\R^\D$-valued Gaussian vector $\eta$ independent of everything else. Then, for every smooth increasing path $L':[0,\T]\to \S^\D_+$, we have\begin{align*}
    L'(0)\cdot \alpha(0)\E_\eta [R(0)] + \int_0^\T L'(t)\cdot \E_\eta[R(t)]\d \alpha(t) = \int_0^\T L'(t)\cdot L_p(t)\d \alpha(t),
\end{align*}
where $\E_\eta$ averages over the randomness of $\eta$.
\end{lemma}

\begin{proof}
Let $\eta'$ be a standard $\R^\D$-valued Gaussian vector independent of everything else. Henceforth, $\E=\E_{\eta,\eta'}$ averages over the randomness of $\eta$ and $\eta'$. Let $q'\in\mcl Q_\infty$ be determined through the relation $\overrightarrow{q'} = L'\circ\alpha^{-1}$ on $(0,1]$. Then, $(L_p, L_q+\eps L',\alpha)$ is a joint decomposition of $(p,q+\eps q')$ for $\eps\geq0$. 
Let $\Phi^\eps$ be associated with $(L_q+\eps L',\alpha)$. Notice that $\sqrt{L_q(0)+\eps L'(0)}\eta$ is equal to $\sqrt{L_q(0)}\eta + \sqrt{\eps L'(0)}\eta'$ in law. Hence, using the first relation in Lemma~\ref{l.R(0),A_0}, we get
\begin{align}\label{e.EPhi^eps=-psi}
    \E \Phi^\eps\Ll(0, \sqrt{2L_q(0)}\eta + \sqrt{2\eps L'(0)}\eta'\Rr) = -\psi \Ll(q+\eps q'\Rr).
\end{align}
Next, we compute the derivatives of both sides in $\eps$ and evaluate them at $\eps=0$.

We first compute the derivative of the left-hand side. For this, we consider two parts
\begin{align*}
    \mathbf{I} & = \eps^{-1}\Ll(\E \Phi^\eps\Ll(0, \sqrt{2L_q(0)}\eta + \sqrt{2\eps L'(0)}\eta'\Rr) - \E \Phi^\eps\Ll(0, \sqrt{2L_q(0)}\eta \Rr)\Rr),
    \\
    \mathbf{II} &= \eps^{-1}\Ll(\E \Phi^\eps\Ll(0, \sqrt{2L_q(0)}\eta \Rr) - \E \Phi^0\Ll(0, \sqrt{2L_q(0)}\eta \Rr)\Rr).
\end{align*}
For the second part, we apply Lemma~\ref{l.d_eps_Phi^eps} (with $\sqrt{2L(0)}\eta$ substituted for $x_0$ therein). We can justify the interchange of derivative and $\E$ using the regularity proved in Proposition~\ref{p.reg_Phi}. Hence, we get
\begin{align*}
    \lim_{\eps\to0}\mathbf{II} = - L'(0)\cdot \E\Ll[A_0+\alpha(0)R(0)\Rr] - \int_0^\T L'(t)\cdot \E [R(t)]\d \alpha(t).
\end{align*}
Now, we turn to $\mathbf{I}$. By Taylor's expansion, we get
\begin{align*}
    \mathbf{I} & = \eps^{-\frac{1}{2}}\E \Ll[\nabla\Phi^\eps(\cdots)\cdot \sqrt{L'(0)}\eta'\Rr] + \E \Ll[\Ll(\sqrt{L'(0)}\eta'\Rr)^\intercal\nabla^2\Phi^\eps(\cdots)\sqrt{L'(0)}\eta'\Rr] + O\Ll(\eps^\frac{1}{2}\Rr)
    \\
    & = L'(0)\cdot \E\Ll[\nabla^2\Phi^\eps(\cdots)\Rr] + O\Ll(\eps^\frac{1}{2}\Rr) 
\end{align*}
where $(\cdots) = \Ll(0,\sqrt{2L_q(0)}\eta\Rr)$. Therefore, 
\begin{align*}
    \lim_{\eps\to0}\mathbf{I}=L'(0)\cdot \E\Ll[\nabla^2\Phi^0(\cdots)\Rr]\stackrel{\eqref{e.R(t)=A_t=}}{=} L'(0)\cdot \E [A_0].
\end{align*}
Putting together the limits of $\mathbf{I}$ and $\mathbf{II}$, we can conclude
\begin{align*}
    \frac{\d}{\d\eps}\Big|_{\eps=0}\E \Phi^\eps\Ll(0, \sqrt{L_q(0)}\eta + \sqrt{\eps L'(0)}\eta'\Rr) =  - L'(0)\cdot \alpha(0)\E\Ll[R(0)\Rr] - \int_0^\T L'(t)\cdot \E [R(t)]\d \alpha(t).
\end{align*}
This gives the derivative of the left-hand side in~\eqref{e.EPhi^eps=-psi}.
For the right-hand side, using the differentiability of $\psi$ as in~\eqref{e.psi_diff}, we have
\begin{align*}
    \frac{\d}{\d\eps}\Big|_{\eps=0} \Ll(-\psi\Ll(q+\eps q'\Rr)\Rr) \stackrel{\eqref{e.psi_diff}}{=} -\int_0^1 q'(s)\cdot p(s) \d s \stackrel{\eqref{e.aq=Lalpha(0,1]}}{=} -\int_0^1 L'\circ\alpha^{-1}(s)\cdot L_p\circ\alpha^{-1}(s) \d s 
    \\
    \stackrel{\text{L.\ref{l.law_dalpha}}}{=} - \int_0^\T L'(t)\cdot L_p(t)\d \alpha(t).
\end{align*}
Combining the two above displays, we get the desired result.
\end{proof}

Recall the definition of pinned decompositions from Subsection~\ref{s.joint_decomp}.

\begin{lemma}[Representation of $\partial_q\psi(q)$, Part~2]\label{l.rep_grad_psi_2}
Under~\ref{i.setting}, assume further that the joint decomposition is pinned and fix $X_0 =0$ (hence, $R(0)$ and $A_0$ are deterministic). 
Then, we have
\begin{gather}
    L_p(t) = R(t),\quad\forall t\in\supp \d\alpha.\label{e.L_p=R}
\end{gather}
\end{lemma}

\begin{proof}
We apply Lemma~\ref{l.rep_grad_psi} to our setting. By taking the difference of two smooth increasing paths $L',L'':[0,\T]\to\S^\D_+$, we can get a smoothed version of $a\one_{[0,t]}$ for every $a\in \S^\D_+$ and $t\in[0,\T]$. Taking the differences of functions of this type, we can get a smoothed version of $a\one_{[t,t']}$ for $a\in \S^\D$ and $0\leq t<t'\leq \T$. This procedure yields a family of test functions rich enough to determine $\S^D$-valued measurable functions on $[0,\T]$.
We refer to this procedure as ``varying $L'$'' in the following.

Since the decomposition is pinned, we have $L_q(0)=0$. Also, now we have set $X_0=0$.
From the definition of $R(\cdot)$ in~\eqref{e.R(t)=A_t=}, we can see that $R(0)$ is deterministic and thus $\E_\eta[R(0)]= R(0)$. 

If $\alpha(0)=0$, then Lemma~\ref{l.rep_grad_psi} gives
\begin{align}\label{e.rel_R(0)=0}
    \int_0^\T L'(t)\cdot R(t)\d \alpha(t) = \int_0^\T L'(t)\cdot L_p(t)\d \alpha(t)
\end{align}
for every smooth increasing $L'$. Varying $L'$, we get~\eqref{e.L_p=R}. If $\alpha(0)>0$, then we have $0\in\supp\d\alpha$ and the mass of $\d\alpha$ at $0$ is exactly $\alpha(0)$. By varying $L'$ to test near $0$, we can get from Lemma~\ref{l.rep_grad_psi} that
\begin{align*}
    \alpha(0)R(0) + \alpha(0)R(0) = L_p(0).
\end{align*}
Since the decomposition is pinned, we have $L_p(0)=0$, which along with $\alpha(0)>0$ implies $R(0)=0$. Inserting this back to the relation in Lemma~\ref{l.rep_grad_psi}, we again get~\eqref{e.rel_R(0)=0} and thus~\eqref{e.L_p=R}.
\end{proof}

\subsection{Left endpoints of paths}

We apply Lemma~\ref{l.rep_grad_psi_2} to studying the left endpoints of $(p,q)$ satisfying $p=\partial_q \psi(q)$.
Recall that, for every $q\in\mcl Q_\infty$, we have fixed $\aq(0)=0$ as in~\eqref{e.aq=}. But due to the second relation in~\eqref{e.aq=} and right-continuous of $q$, we have
\begin{align}\label{e.q(0)=aq(0+)}
    q(0)= \aq(0+),\quad\forall q \in \mcl Q_\infty.
\end{align}
Hence, to find the value of $q(0)$, it is equivalent to determine $\aq(0+)$ and vice versa.

To describe the result in the next lemma, we need to define the following object. For every $q\in\mcl Q_\infty$, we define
\begin{align}\label{e.V(q)=}
    V(q)= \int_0^1 \E_\eta\Ll[a(s,\eta)q(0)a(s,\eta)^\intercal\Rr]\d s
\end{align}
where $\eta$ is a standard $\R^\D$-valued Gaussian vector independent of everything else and
\begin{align}\label{e.a(s,eta)=}
    a(s,\eta)= \E \la \sigma\sigma^\intercal -\sigma\sigma'^\intercal \ra_{P_1,\, q-q(0)s,\,\sqrt{2q(0)s}\eta,\, -q(0)s},\quad\forall s\in[0,1].
\end{align}
Here, the Gibbs measure is given as in~\eqref{e.f(q),<>_q=}.
Notice that $V(q)=0$ if $q(0)=0$.

\begin{lemma}[Formula for the left endpoint]\label{l.p(0)=mm+2V}
Let $p,q\in\mcl Q_\infty$ satisfy $p= \partial \psi(q)$, let $V(q)$ be given as in~\eqref{e.V(q)=}, and let $\la\cdot\ra_q$ be given as in~\eqref{e.<>_q_correct}.
We have
\begin{align}\label{e.p(0)=...}
    p(0)= \E \la \sigma\ra_{q}\E \la \sigma\ra_{q}^\intercal + 2V(q).
\end{align}
In particular, when $p(0)=q(0)=0$, we have $\E \la \sigma\ra_{q}=0$.
\end{lemma}

\begin{proof}
Let $(L_p,L_q,\alpha)$ be the canonical decomposition of $(p,q)$ given as in~\eqref{e.alpha=_joint_decomp} and~\eqref{e.L_k=_joint_decomp}. 
We also assume the setting~\ref{i.setting} and fix $X_0=0$.
We consider two cases: $p(0)=q(0)=0$ or otherwise.

In the first case, we have $\ap(0+)=\aq(0+)=0$, which by $\alpha^{-1}= \tr \ap + \tr \aq$ (see~\eqref{e.alpha=_joint_decomp}) implies $\alpha^{-1}(0+)=0$. By the characterization of $\supp\d\alpha$ in Lemma~\ref{l.law_dalpha}, we have $0\in\supp\d\alpha$, which allows us to apply Lemma~\ref{l.rep_grad_psi_2} at $t=0$. Since the canonical decomposition is pinned, we have $L_p(0)=0$. Hence, we get
\begin{align*}
    0=L_p(0) \stackrel{\text{L.\ref{l.rep_grad_psi_2}}}{=} R(0) \stackrel{\text{L.\ref{l.R(0),A_0}}}{=} \E \la \sigma\ra_{q}\E \la \sigma\ra_{q}^\intercal,
\end{align*}
which implies $\E \la \sigma\ra_{q}=0$ and thus~\eqref{e.p(0)=...} holds in this case.

Now, we consider the other case when either $p(0)\neq 0$ or $q(0)\neq 0$. Set $t=\alpha^{-1}(0+)$, which lies in $\supp\d\alpha$. By the relation $\alpha^{-1} = \tr\ap + \tr \aq$ and~\eqref{e.q(0)=aq(0+)}, we have $t>0$.
Using the property~\eqref{e.aq=Lalpha(0,1]} of decompositions, we get $\ap(0+)=L_p(t)$ and $\aq(0+)= L_q(t)$.
Applying Lemma~\ref{l.rep_grad_psi_2}, \eqref{e.R(t)=R(0)+2int} in Lemma~\ref{l.R,A_rel}, and the last statement in Lemma~\ref{l.R(0),A_0}, we get
\begin{align*}
    p(0) \stackrel{\eqref{e.q(0)=aq(0+)}}{=} \ap(0+)= L_p(t) = R(t)= \E \la \sigma\ra_{q}\E \la \sigma\ra_{q}^\intercal + 2 \int_0^t \E \Ll[A_r^\intercal \dot L_q(r) A_r\Rr]\d r.
\end{align*}
The proof is complete if we can show
\begin{align}\label{e.int...=V(q)}
    \int_0^t \E \Ll[A_r^\intercal \dot L_q(r) A_r\Rr]\d r = V(q).
\end{align}
We devote the remainder of the proof to verifying this relation.

We start by identifying $\dot L_q$ on $[0,t]$. Since the canonical decomposition is pinned, we have $L_q(0)=0$. Due to $\alpha^{-1}(0+)=t$, we can use Lemma~\ref{l.alpha(t)=sup} to see $\alpha(t)=0$ and thus $\alpha(r)=0$ on $[0,t]$ via monotonicity. As a result, $(0,t)\cap \supp\d\alpha=\emptyset$. Due to the definition of the canonical decomposition in~\eqref{e.L_k=_joint_decomp}, we have that $L_q$ on $[0,t]$ is a linear interpolation between $L_q(0)$ and $L_q(t)= q(0)$. Hence, 
\begin{align}\label{e.L_q(r)=...}
    L_q(r)=q(0)r/t,\qquad \dot L_q(r) = q(0)/t,\qquad \forall r\in(0,t).
\end{align}
Now, the process $X$ given as in~\eqref{e.X_t=} satisfies $\d X_r = \sqrt{2q(0)/t}\d W_r$ on $[0,t]$. Since we have fixed $X_0=0$, we have
\begin{align}\label{e.X(r)=...}
    X_r \stackrel{\d}{=} \sqrt{2q(0)r/t}\eta,\qquad\forall r\in[0,t].
\end{align}

Next, we study $\nabla^2\Phi(r,\cdot)$ for $r\in[0,t]$ by using the representation in Lemma~\ref{l.cascade_rep}. 
Let $\pi_r$ and $\alpha_{[r}$ be given as in that lemma.
Due to $\alpha=0$ on $[0,t]$, we have $\alpha_{[r}= \alpha$ for every $r\in[0,t]$ and thus $\pi_r = L_q\circ\alpha^{-1}-L_q(r) = \aq - \aq(0)r/t$ on $(0,t]$, where we used $L_q\circ\alpha^{-1}=\aq$ on $(0,1]$ due to~\eqref{e.aq=Lalpha(0,1]}.
Also, recall the first relation in~\eqref{e.L_q(r)=...} and the definition of the Gibbs measure in~\eqref{e.f(q),<>_q=}. Therefore, the natural Gibbs measure associated with the representation of $\Phi(r,x)$ for $r\in[0,t]$ given by Lemma~\ref{l.cascade_rep} is $\la\cdot\ra_{P_1,\, q-q(0)r/t,\, x,\, -q(0)r/t}$. Hence, we can compute
\begin{align*}
    \nabla^2\Phi(r,x) = \E \la \sigma\sigma^\intercal - \sigma\sigma'^\intercal\ra_{P_1,\, q-q(0)r/t,\, x,\, -q(0)r/t},\quad\forall (r,x)\in[0,t]\times \R^\D.
\end{align*}
Recall the definition of $A_\cdot$ in~\eqref{e.R(t)=A_t=}. Using the above display and~\eqref{e.X(r)=...}, we can thus verify that $A_r \stackrel{\d}{=} a(r/t,\eta)$ for $a(\cdot,\cdot)$ given in~\eqref{e.a(s,eta)=}. This along with~\eqref{e.L_q(r)=...} gives
\begin{align*}
    \int_0^t \E \Ll[A_r^\intercal \dot L_q(r) A_r\Rr]\d r = \frac{1}{t}\int_0^t \E_\eta \Ll[a(r/t,\eta)^\intercal q(0)a(r/t,\eta)\Rr]\d r,
\end{align*}
which is equal to $V(q)$ given in~\eqref{e.V(q)=} after a change of variable. This verifies~\eqref{e.int...=V(q)} and completes the proof.
\end{proof}

We can view $\E \la \sigma\ra_{q}$ as the mean magnetization associated with the Gibbs measure $\la\cdot\ra_q$. If $p$ represents the Parisi measure, then $p(0)$ is the smallest point in its support. The following corollary is an easy consequence of Lemma~\ref{l.p(0)=mm+2V}. It states that if the mean magnetization under $\la \cdot\ra_{q+t\nabla\xi(p)}$ is not zero, then $0$ is not in the support of the Parisi measure. Recall that we know \cite{auffinger2015properties} in the setting with Ising spins and no external field that $0$ is always in the support of the Parisi measure (there, because of $D=1$ and the Ising setup, $\E\la \sigma\ra_q=0$ always holds). 

\begin{corollary}
Let $p,q\in\mcl Q_\infty$ satisfy $p=\partial_q \psi(q+t\nabla\xi(p))$ for some $t\geq 0$ and $\xi$. Assume $\nabla\xi(0)=0$ and $q(0)=0$. 
If $\E \la\sigma\ra_{q+t\nabla\xi(p)} \neq 0$, then $p(0)\neq 0$.
\end{corollary}

\subsection{Detection of increments}
\label{s.detect}

In the previous subsection, through Lemma~\ref{l.p(0)=mm+2V}, we have gained understanding of $p(0)$ and $q(0)$. Now, we turn to investigating the increments of the two paths at $s\in(0,1)$.

In the general vector case $D\geq 1$, we need to first describe non-trivial directions of the spin distributing $P_1$ and clarify its relation to the strict convexity of $\Phi$.

\begin{lemma}[Non-trivial directions along the Hessian]\label{l.convex_Phi}
Let $(L,\alpha)$ be a decomposition of some $q\in \mcl Q_\infty$, let $\Phi$ be associated with $(L,\alpha)$, and let $\phi$ be the terminal condition of $\Phi$ given as in~\eqref{e.phi=initial}.
Then, the following holds:
\begin{enumerate}
    \item \label{i.P_1_not_Dirac} If $P_1$ is not a Dirac mass, then there exists $z\in\R^\D$ such that
    \begin{align}\label{e.znabla^2Phi(x)z>0}
        z^\intercal \nabla^2 \phi(x)z >0,\qquad\forall x\in\R^\D.
    \end{align}

    \item\label{i.P_1_span}

    If $\supp P_1$ spans $\R^\D$, then~\eqref{e.znabla^2Phi(x)z>0} holds for every non-zero $z\in \R^\D$.
    
    \item \label{i.l.convex_Phi_2} For every $z\in \R^\D$ satisfying~\eqref{e.znabla^2Phi(x)z>0}, it holds that $z^\intercal \nabla^2 \Phi(t,x)z >0$ for every $(t,x)\in [0,1]\times\R^\D$.
\end{enumerate}
\end{lemma}

\begin{proof}
In view of the definition of $\phi$ in~\eqref{e.phi=initial}, we can compute, for $x,z\in\R^\D$,
\begin{align*}
    z\cdot \nabla^2\phi(x)z = \frac{\d^2}{\d\eps^2}\Big|_{\eps=0}\phi(x+\eps z) = \la \Ll(\sigma\cdot z -\la \sigma\cdot z\ra_x\Rr)^2 \ra_x
\end{align*}
where $\la\cdot\ra_x \propto \exp\Ll(\sigma\cdot x - \sigma\sigma^\intercal\cdot L(T)\Rr)\d P_1(\sigma)$. To prove~\eqref{i.P_1_not_Dirac}, we argue by contradiction and suppose that for every $z$ there is $x_z$ such that $z^\intercal \nabla^2\phi(x_z)z =0$. Then, the above display implies that $\sigma\cdot z = \la\sigma\cdot z\ra_{x_z}$ a.s.\ under $\la\cdot\ra_{x_z}$ which thus also holds a.s.\ under $P_1$. Varying $z$, we deduce that $\sigma$ is constant a.s.\ under $P_1$ reaching a contradiction. Hence, the first part is verified.

For Part~\eqref{i.P_1_span}, we again argue by contradiction and suppose that there are $z$ and $x$ such that $z^\intercal \nabla^2\phi(x_z)z =0$. Again, we get $\sigma\cdot z = \la\sigma\cdot z\ra_x$ a.s.\ under $P_1$. Hence, $\supp P_1$ is contained in an affine plane, reaching a contradiction.

For Part~\eqref{i.l.convex_Phi_2}, we fix any $(t,x)$. We assume the setting~\ref{i.setting} but
redefine $(X_r)_{r\in[t,\T]}$ to be the strong solution of the SDE in~\eqref{e.X_t=} but with initial condition at $t$ satisfying $X_t=x$. Then computation in the proof of Lemma~\ref{l.R,A_rel} still holds. Combining~\eqref{e.R(t)=R(0)+2int} and~\eqref{e.E[nabla^2Phi]=}, we get
\begin{align*}
    \E [A_t]=\E[A_\T] + 2\int_t^\T\alpha(r)\E \Ll[A^\intercal_r\dot L(r)A_r\Rr]\d r \geq \E [A_\T]
\end{align*}
where $A_\cdot$ is defined as in~\eqref{e.R(t)=A_t=} but with the new $X$. This implies $\nabla^2\Phi(t,x)\geq \E \Ll[\nabla^2\phi(X_\T)\Rr]$, which along with~\eqref{e.znabla^2Phi(x)z>0} yields the desired result.
\end{proof}

We need a simple lemma on matrices.

\begin{lemma}\label{l.A>0=>A>zz}
For $A\in\S^\D_+$, we have $A\neq 0$ if and only if $A\geq zz^\intercal$ for some nonzero $z\in \R^\D$.
\end{lemma}

\begin{proof}
Since $A \in \S^\D_+$, there exist $\lambda_1 \ge \cdots \ge \lambda_D \ge 0$ and $(v_1, \ldots, v_D)$ an orthonormal basis of $\R^D$ such that $A = \sum_{d = 1}^D \lambda_d v_d v_d^\intercal$. If $A \neq 0$, then $\lambda_1 \neq 0$ and we can choose $z = \sqrt{\lambda_1} v_1$. The converse implication is obvious. \end{proof}

With the above preparation, we are ready prove results on detecting increments of paths.

\begin{lemma}[Detection of a general increment]\label{l.detect_gen_incre}
Let $p,q\in\mcl Q_\infty$ satisfy $p=\partial_q \psi(q)$, assume that $\supp P_1$ spans $\R^\D$, and let $0<s<s'<1$. We have that $\ap(s')\neq \ap(s)$ if and only if $\aq(s')\neq \aq(s)$.
\end{lemma}

\begin{proof}
Instead of the original statement, we show that $\ap(s')=\ap(s)$ if and only if $\aq(s')=\aq(s)$.
Before proceeding into the proof, we need some preparation.
Let $(L_p,L_q,\alpha)$ be the canonical joint decomposition of $p$ and $q$ given in~\eqref{e.alpha=_joint_decomp} and~\eqref{e.L_k=_joint_decomp}. 
Adopt the setting~\ref{i.setting} for the pair $(p,q)$.
By Lemma~\ref{l.law_dalpha}, we have $\alpha^{-1}(s),\alpha^{-1}(s')\in\supp\d\alpha$.
Applying Lemma~\ref{l.rep_grad_psi_2} to the pair $(p,q)$ and Lemma~\ref{l.R,A_rel}, we get 
\begin{align}\label{e.p(s')-p(s)=...=}
    \ap(s')-\ap(s) \stackrel{\eqref{e.aq=Lalpha(0,1]}}{=} L_p\circ\alpha^{-1}(s') - L_p\circ\alpha^{-1}(s) =2\int_{\alpha^{-1}(s)}^{\alpha^{-1}(s')} \E \Ll[A_r^\intercal \dot L_q(r) A_r\Rr]\d r.
\end{align}
By~\eqref{e.aq=Lalpha(0,1]}, we also have
\begin{align}\label{e.q(s')-q(s)=L...}
    \aq(s')-\aq(s) = L_q\circ\alpha^{-1}(s') - L_q\circ\alpha^{-1}(s).
\end{align}

First, assume that $\aq(s')=\aq(s)$. Then, from~\eqref{e.q(s')-q(s)=L...}, we can deduce that either $\alpha^{-1}(s')=\alpha^{-1}(s)$ or $\dot L_q =0$ a.e.\ on $\Ll[\alpha^{-1}(s),\,\alpha^{-1}(s')\Rr]$. In either case, \eqref{e.p(s')-p(s)=...=} implies $\ap(s')=\ap(s)$.

Next, assume $\ap(s')=\ap(s)$. If $\alpha^{-1}(s')=\alpha^{-1}(s)$, then we have $\aq(s')=\aq(s)$ from~\eqref{e.q(s')-q(s)=L...} as desired. Hence, we further assume $\alpha^{-1}(s')>\alpha^{-1}(s)$ and show $\dot L_q =0$ a.e.\ on $\Ll[\alpha^{-1}(s),\,\alpha^{-1}(s')\Rr]$. Fix a sequence $(x_n)_{n\in\N}$ in $\R^\D$ that forms a dense subset of $\R^\D$. The assumption $\ap(s')=\ap(s)$ together with~\eqref{e.p(s')-p(s)=...=} implies
\begin{align*}
    \int_{\alpha^{-1}(s)}^{\alpha^{-1}(s')}  x_n^\intercal\E\Ll[A_r^\intercal \dot L_q(r) A_r\Rr]x_n\d r =0,\quad\forall n\in\N.
\end{align*}
Write $I=\Ll[\alpha^{-1}(s),\,\alpha^{-1}(s')\Rr]$. For each $n\in\N$, let $I_n$ be the measurable subset of $I$ satisfying $\Leb(I\setminus I_n)=0$ and $x_n^\intercal\E\Ll[A_r^\intercal \dot L_q(r) A_r\Rr]x_n=0$ for every $r\in I_n$. Taking $I_\star = \bigcap_{n\in\N}I_n$, we have $\Leb(I\setminus I_\star) =0$. Fix any $r\in I_\star$. For every $n$, let $\Omega_{r,n}$ be the full measure set on which $x_n^\intercal A_r^\intercal \dot L_q(r) A_r x_n=0$ (we can do so because this quantity is nonnegative). Take $\Omega_{r,\star} =\bigcap_{n\in\N}\Omega_{r,n}$ which is still a full measure set. The density of $(x_n)_{n\in\N}$ implies that, on $\Omega_{r,\star}$, we have
$x^\intercal A_r^\intercal \dot L_q(r) A_r x=0$ for every $x\in \R^\D$. Recall the definition of $A_r$ from~\eqref{e.R(t)=A_t=}. The assumption on $\supp P_1$ allows us to apply Lemma~\ref{l.convex_Phi} to see that $A_r\in\S^\D_{++}$ and thus $A_r$ is an invertible matrix a.s.
Hence, we can assume that $A_r$ is invertible on $\Omega_{r,\star}$. Fixing any realization from $\Omega_{r,\star}$ and varying $x\in\R^\D$, we deduce that $y^\intercal \dot L_q(r) y=0$ for every $y \in \R^D$, which implies $L_q(r)=0$ for $r\in I_\star$. Inserting this back to~\eqref{e.p(s')-p(s)=...=}, we see that the integral therein is zero and thus $\ap(s')=\ap(s)$. This completes the proof.
\end{proof}

Next, we give a formula for the increment at discontinuities.

\begin{lemma}[Formula at a discontinuity point]\label{l.jump_formula}

Under~\ref{i.setting}, further assume that the decomposition is canonical and fix $X_0=0$.
If either $\aq(s+)\neq\aq(s)$ or $\ap(s+)\neq\ap(s)$ at some $s\in (0,1)$, then there are $t,t_\star\in\supp\d\alpha$ satisfying~\eqref{e.t,t_star_prop} and
\begin{gather*}
    \ap(s+)-\ap(s) = \frac{2}{t_\star-t}\int_t^{t_\star} \E \Ll[A_r^\intercal \Ll(\aq(s+)-\aq(s)\Rr)A_r\Rr] \d r .
\end{gather*}

\end{lemma}

\begin{proof}
Due to $\ap=L_p\circ\alpha^{-1}$ and $\aq=L_q\circ\alpha^{-1}$ on $(0,1]$ (see~\eqref{e.aq=Lalpha(0,1]}), the assumption on the discontinuity of either of them at $s$ implies $\alpha^{-1}(s+) - \alpha^{-1}(s)>0$. By Lemma~\ref{l.equival_jumps}, there are $t,t_\star\in\{0\}\cup\supp\d\alpha$ such that~\eqref{e.t,t_star_prop} is satisfied. Here, $t,t'\in\supp\d\alpha$ because $s>0$ (see Lemma~\ref{l.law_dalpha}). 
In particular, we have $t=\alpha^{-1}(s)$, $t_\star = \alpha^{-1}(s+)$. Using this, $\ap=L_p\circ\alpha^{-1}$ and $\ap=L_q\circ\alpha^{-1}$ on $(0,1]$, we get
\begin{align*}
    \ap(s+)=L_p(t_\star),\qquad \ap(s)=L_p(t),\qquad \aq(s+)=L_q(t_\star),\qquad \aq(s)=L_q(t).
\end{align*}
Using Lemma~\ref{l.rep_grad_psi_2} and Lemma~\ref{l.R,A_rel}, we can get
\begin{align*}
     \ap(s+)-\ap(s)= R(t_\star) - R(t)= 2\int_t^{t_\star}\E \Ll[A_r^\intercal \dot L_q(r)A_r\Rr] \d r. \end{align*}
Next, we determine $\dot L_q(r)$ for $r\in[t,t_\star]$. Recall the definition of $L_q$ in~\eqref{e.L_k=_joint_decomp}. Since $t,t_\star\in \supp\d\alpha$ and $(t,t_\star)\cap\supp\d \alpha=\emptyset$, we can see that $L_q$ on $(t,t_\star)$ is defined as a linear interpolation from $L_q(t)$ to $L_q(t_\star)$. Hence, we have
\begin{align*}
    \dot L_q(r) = \frac{L_q(t_\star)-L_q(t)}{t_\star-t} = \frac{\aq(s+)-\aq(s)}{t_\star-t},\quad\forall r\in(t,t_\star).
\end{align*}
Inserting this to the above display, we get the desired result.
\end{proof}

Using the above lemma, we are able to detect discontinuity points of paths.

\begin{corollary}[Detection of discontinuity]\label{c.jump_iff}
Let $p,q\in\mcl Q_\infty$ satisfy $p=\partial_q\psi(q)$ and let $s\in(0,1)$. If $\ap(s+)\neq\ap(s)$, then $\aq(s+)\neq\aq(s)$. Conversely, if $\aq(s+)-\aq(s)\geq zz^\intercal$ for some $z\in\R^\D$ satisfying~\eqref{e.znabla^2Phi(x)z>0}, then $z^\intercal \Ll(\ap(s+)-\ap(s)\Rr)z>0$. 

Under the assumption that $\supp P_1$ spans $\R^\D$, we have $\ap(s+)\neq\ap(s)$ if and only if $\aq(s+)\neq\aq(s)$.
\end{corollary}

\begin{proof}
If $\aq(s+)=\aq(s)$, then Lemma~\ref{l.jump_formula} implies $\ap(s+)=\ap(s)$, which verifies the first implication. Now assuming $\aq(s+)-\aq(s)\geq zz^\intercal$, we can use the formula in Lemma~\ref{l.jump_formula} to set
\begin{gather*}
    z^\intercal \Ll(\ap(s+)-\ap(s)\Rr)z = \frac{2}{t_\star-t}\int_t^{t_\star} \E \Ll[(A_r z)^\intercal \Ll(\aq(s+)-\aq(s)\Rr)A_r z \Rr] \d r 
    \\
    \geq \frac{2}{t_\star-t}\int_t^{t_\star} \E \Ll[(z^\intercal A_r z)^2 \Rr] \d r 
\end{gather*}
which is strictly positive due to Lemma~\ref{l.convex_Phi}~\eqref{i.l.convex_Phi_2} and the definition of $A_\cdot$ in~\eqref{e.R(t)=A_t=}. We turn to the last assertion and there is only one direction to be verified. If $\aq(s+)\neq\aq(s)$, Lemma~\ref{l.A>0=>A>zz} implies $\aq(s+)-\aq(s)\geq zz^\intercal$ for some nonzero $z\in\R^\D$. 
The assumption on $\supp P_1$ together with Lemma~\ref{l.convex_Phi}~\eqref{i.P_1_span} ensures that $z$ satisfies~\eqref{e.znabla^2Phi(x)z>0} and the desired result follows.
\end{proof}

Finally, we give a formula for derivatives of paths on an interval on which they are absolutely continuous. Due to the regularity assumption on $\alpha^{-1}$, $L_p$, and $L_q$, this result is not used in this paper but can be useful in specific scenarios.

\begin{lemma}[Formula along an absolutely continuous increment]\label{l.formula_cts_incre}

Under~\ref{i.setting}, fix $X_0=0$. 
If $\alpha^{-1}$ is absolutely continuous on some interval $I\subset (0,1)$ and $L_p$ and $L_q$ are differentiable on $\Ll\{\alpha^{-1}(s):\:s\in I\Rr\}$, then $\ap$ and $\aq$ are differentiable at almost every $s\in I$ and satisfy
\begin{align}\label{e.dotap(s)=2E[Adotaq(s)A]}
    \dot \ap(s) = 2\E \Ll[A_{\alpha^{-1}(s)}^\intercal \dot \aq(s)A_{\alpha^{-1}(s)}\Rr].
\end{align}
\end{lemma}

\begin{proof}
Set $J=\Ll\{\alpha^{-1}(s):\:s\in I\Rr\}$.
If $\alpha^{-1}$ is constant on $I$, then both $\ap$ and $\aq$ are constant on $I$ due to~\eqref{e.aq=Lalpha(0,1]}. In this case, \eqref{e.dotap(s)=2E[Adotaq(s)A]} holds trivially. Henceforth, we assume that $\alpha^{-1}$ is not constant on $I$ and thus $J$ is not a singleton.
Due to $I\subset (0,1)$, Lemma~\ref{l.law_dalpha} implies $J\subset \supp \d \alpha$. This allows us to apply Lemma~\ref{l.rep_grad_psi_2} and Lemma~\ref{l.R,A_rel} to get
\begin{align*}
     L_p(t')-L_p(t) = 2\int_t^{t'}\E \Ll[A_r^\intercal \dot L_q(r)A_r\Rr] \d r, \quad\forall t,t'\in J.
\end{align*}
Since $L_p$ and $L_q$ are differentiable on $J$, we have
\begin{align}\label{e.dotL_p(t)=2E[AdotL_q]}
    \dot L_p(t) = 2 \E \Ll[A_t^\intercal \dot L_q(t)A_t\Rr],\quad\forall t\in J.
\end{align}
Since $\alpha^{-1}$ is absolutely continuous, there is a measurable subset $I_0$ of $I$ with $\Leb(I\setminus I_0)=0$ such that $\alpha^{-1}$ is differentiable on $I_0$. Using~\eqref{e.aq=Lalpha(0,1]}, we have $\dot \ap(s) = \dot L_p(\alpha^{-1}(s))\dot \alpha^{-1}(s)$ and $\dot \aq(s) = \dot L_q(\alpha^{-1}(s))\dot \alpha^{-1}(s)$ for every $s\in I$. These along with the above display yields the desired result.
\end{proof}

We briefly explain the reason for requiring $L_p$ and $L_q$ to be differentiable in Lemma~\ref{l.formula_cts_incre}. Without this assumption, we can still use the Lipschitzness and the Lebesgue differentiation theorem to get the relation in~\eqref{e.dotL_p(t)=2E[AdotL_q]} but only for $t$ in a full measure subset $J_0$ of $J$. To proceed, we need to change variable from $t$ to $\alpha^{-1}$ and ensure that $\Ll\{s\in I:\: \alpha^{-1}(s)\in J_0\Rr\}$ is a full measure subset of $I$. But in general this does not hold. 

We briefly mention a concrete situation where this lemma can be used. Recall the critical point relation~\eqref{e.crit.point} and set $q=0$ therein. Also set $D=1$. In this case, $\ap$ is a real-valued path and we can simply take $L_p(t)=t$ for each $t$ and $\alpha^{-1}=\ap$. Then, $(L_p, t\nabla\xi \circ L_p, \alpha)$ is a decomposition of $(p,t\nabla\xi(p))$. We can apply the above lemma to this pair since both $L_p$ and $t\nabla\xi \circ L_p$ are differentiable everywhere (assuming even regularity of $\xi$).

\subsection{Analysis of the replica-symmetry breaking structure}

Recall the critical-point condition in \eqref{e.crit.point}, which can be rewritten as $p=\partial_q \psi(q+t\nabla\xi(p))$. We now transfer the results of the previous subsection to this setting.
\begin{proposition}[RSB induced by an external field]Let $p,q\in\mcl Q_\infty$ satisfy $p=\partial_q \psi(q+t\nabla\xi(p))$ for some $t\geq 0$ and $\xi$. Assume that $\supp P_1$ spans $\R^\D$. 
We have that, for $0<s<s'<1$, if $\aq(s')\neq \aq(s)$, then $\ap(s')\neq \ap(s)$.

Also, if $\aq(s+)\neq \aq(s)$ at some $s\in (0,1)$, then $\ap(s+)\neq \ap(s)$. More precisely, if $\aq(s+)- \aq(s)\geq zz^\intercal$ for some $z\in\R^\D\setminus\{0\}$, then $z^\intercal \Ll(\ap(s+)-\ap(s)\Rr)z>0$.

\end{proposition}

\begin{proof}
Write $q' = q+t\nabla\xi(p)$.
The main statement follows from Lemma~\ref{l.detect_gen_incre} applied to the pair $(p,q')$.
The additional statement follows from Corollary~\ref{c.jump_iff} applied $(p,q')$. 
\end{proof}

The mechanism behind the simultaneity of the replica-symmetry-breaking structures between the different types of spins  is based on the fact that the function $\nabla \xi$ can transfer an increment along some direction $y\in\R^\D$ to some other direction $z\in\R^\D$. To formulate this precisely, for nonzero $y,z\in\R^\D$, we say that $\nabla\xi$ is \textbf{$y$-to-$z$ coupled} if for every $a\in\S^\D_+$ and $b\in\S^\D_+$ satisfying $y^\intercal by>0$, there is $c>0$ such that
\begin{align}\label{e.strongly_ytoz}
    \frac{\d}{\d \eps}\Big|_{\eps=0}\nabla\xi(a+\eps b)\geq c zz^\intercal.
\end{align}

\begin{proposition}[Simultaneous RSB]\label{p.sim_RSB}
Let $t > 0$ and $p,q\in\mcl Q_\infty$ satisfy $\partial_q \psi(q+t\nabla\xi(p)) =p$, let $y,z \in \R^D \setminus \{0\}$, and assume that $\supp P_1$ spans $\R^\D$ and that $\nabla \xi$ is $y$-to-$z$ coupled. 

\begin{itemize}
\item
If $y^\intercal \ap(s') y >y^\intercal \ap(s) y$ for some $0< s<s'<1$, then $z^\intercal \ap(s') z >z^\intercal \ap(s) z$.
\item
If $y^\intercal \ap(s+) y >y^\intercal \ap(s) y$ for some $0< s < 1$, then $z^\intercal \ap(s+) z >z^\intercal \ap(s) z$.
\end{itemize}
\end{proposition}

\begin{proof}
We write $q' = q + t \nabla \xi(p)$, and start by proving the second statement. If $y^\intercal \ap(s+) y >y^\intercal \ap(s) y$, then the assumption that $\xi$ is $y$-to-$z$ coupled guarantees the existence of a constant $c > 0$ such that $\nabla \xi(\ap(s+))- \nabla \xi(\ap(s)) \ge c z z^\intercal$, and thus $\overrightarrow{q'}(s+) - \overrightarrow{q'}(s) \ge tc z z^\intercal$. An application of Corollary~\ref{c.jump_iff} to the pair $(p,q')$ thus yields the conclusion.

We now turn to the first part of the statement. Let $(L_p,L_q,\alpha)$ be the canonical joint decomposition of $p$ and $q$ given in~\eqref{e.alpha=_joint_decomp} and~\eqref{e.L_k=_joint_decomp}. We adopt the setting~\ref{i.setting} accordingly. 
Due to $\alpha^{-1}=\tr \ap+\tr\aq$ (see~\eqref{e.alpha=_joint_decomp}) and the assumption $\ap(s')\neq \ap(s)$, we have $\alpha^{-1}(s')>\alpha^{-1}(s)$.
Setting $L_{q'} = L_q + t\nabla\xi\circ L_p$, we have $\overrightarrow{q'} = L_{q'}\circ\alpha^{-1}$.
Applying Lemma~\ref{l.rep_grad_psi_2} to the pair $(p,q')$ and Lemma~\ref{l.R,A_rel}, we get
\begin{align}\label{e.p(s')-p(s)=2intE}
    \ap(s')-\ap(s) \stackrel{\eqref{e.aq=Lalpha(0,1]}}{=} L_p\circ\alpha^{-1}(s') - L_p\circ\alpha^{-1}(s) =2\int_{\alpha^{-1}(s)}^{\alpha^{-1}(s')} \E \Ll[A_r^\intercal \dot L_{q'}(r) A_r\Rr]\d r.
\end{align}
Here, $\alpha^{-1}(s),\alpha^{-1}(s')\in\supp\d\alpha$ due to $s,s'>0$ and Lemma~\ref{l.law_dalpha}. We set
\begin{align*}
    G=\Ll\{r\in\Ll[\alpha^{-1}(s),\,\alpha^{-1}(s')\Rr]:\:y^\intercal \dot L_p(r)y\neq 0\Rr\}.
\end{align*}
Due to the condition $y^\intercal\ap(s')y\neq y^\intercal\ap(s)y$, we must have $\Leb\Ll(G\Rr)>0$. 
Using~\eqref{e.strongly_ytoz}, we can see that for every $ r\in G$ there is $c(r)>0$ such that $\frac{\d}{\d r}\nabla\xi \Ll(L_p(r)\Rr)\geq c(r)zz^\intercal$.
We can ensure that $c(\cdot)$ is a measurable function by choosing $c(r) = \inf \frac{u^\intercal \dot L_{q'}(r) u}{u^\intercal z z^\intercal u}$ for $r\in G$, where the infimum is taken over a countable dense subset of $\Ll\{u\in\R^\D:z^\intercal u >0\Rr\}$. 
By giving up a subset of zero measure, we can assume that $L_q$ and $L_{q'}$ are differentiable on $G$. Then, we have $\dot L_{q'}(r)=\dot L_q(r) + t\frac{\d}{\d r}\nabla\xi \Ll(L_p(r)\Rr) \geq t c(r)zz^\intercal$ for each $r\in G$, and since $\dot{L}_q(r) \ge 0$, we deduce that $\dot L_{q'}(r) \ge t c(r) z z^\intercal$. 
Inserting this to~\eqref{e.p(s')-p(s)=2intE}, we get
\begin{align*}
    z^\intercal \Ll(\ap(s')-\ap(s) \Rr)z \geq 2t \int_{\alpha^{-1}(s)}^{\alpha^{-1}(s')}  c(r)\E \Ll[\Ll|A_r z\Rr|^2\Rr]\d r.
\end{align*}
By Lemma~\ref{l.convex_Phi}~\eqref{i.l.convex_Phi_2}, the right-hand side is strictly positive, which yields the desired result.
\end{proof}

\begin{proof}[Proof of Theorem~\ref{t.main}]
We first show the result with $p$ replaced by its left-continuous version $\ap$. 
Assumption~\eqref{e.ass.coupling} implies that $\nabla\xi$ is $y$-to-$z$ coupled for every nonzero $y,z\in\R^\D$, in the sense of~\eqref{e.strongly_ytoz}. Now, if $\ap(s')-\ap(s)\neq 0$ for $0<s<s'<1$, then there is nonzero $y\in\R^\D$ such that $y^\intercal \ap(s')y > y^\intercal \ap(s) y$. The first part of Proposition~\ref{p.sim_RSB} implies $z^\intercal\ap(s') z>z^\intercal \ap (s)z$ for every nonzero $z\in\R^\D$, which gives $\ap(s')-\ap(s)\in\S^\D_{++}$ as desired.

We now show that the same statement is also valid with the right-continuous version of the path. Suppose that $y^\intercal p(s')y > y^\intercal p(s) y$ for $0\leq s<s'\leq 1$. If the function $u \mapsto y^\intercal p(u)y$ is non-constant on $(s,s')$, then we can apply the result that concerns the left-continuous path $\ap$ and obtain the conclusion. Otherwise, the function $u \mapsto y^\intercal p(u)y$ must have a jump at $u = s'$ (a jump at $u = s$ is not possible by right-continuity) and we must have $s'\neq 1$ (due to the continuity of $p$ at $1$ by definition as in~\eqref{e.q(1)=}). This means that we have $y^\intercal \ap(s'+)y > y^\intercal \ap(s') y$. We can therefore appeal to the second part of Proposition~\ref{p.sim_RSB} and argue as in the previous paragraph to conclude that $\ap(s'+) -  \ap(s') \in S^D_{++}$. This implies that $p(s') - p(s) \in S^D_{++}$, as desired. 
\end{proof}

\begin{remark}
\label{r.example}
It was stated in the introduction that the presence of a term of the form of \eqref{e.extra.energy} is sufficient to guarantee the validity of the assumption \eqref{e.ass.coupling}. We clarify why this is so here, and for convenience we simply fix $D = 2$ and choose the energy function $H_N(\sigma)$ to be given by
\begin{align}
\label{e.example}
    H_N(\sigma) = N^{-\frac{1}{2}}\sum_{i,j=1}^N \Ll(g^{11}_{ij}\sigma_{1i}\sigma_{1j}+g^{12}_{ij}\sigma_{1i}\sigma_{2j}+g^{22}_{ij}\sigma_{2i}\sigma_{2j}\Rr),\quad\forall \sigma\in \R^{2\times N},
\end{align}
where $(g^{kl}_{ij})_{1\leq k,l\leq 2;\; i,j\ge 1}$ are independent standard Gaussian variables. The covariance of this energy function is given as in~\eqref{e.def.xi} for the function $\xi:\R^{2\times 2}\to\R$ such that $\xi(a) = a^2_{11}+a^2_{22}+a_{11}a_{22}$, where we write $a=(a_{ij})_{1\leq i,j\leq 2}\in\R^{2\times2}$. We have
\begin{align*}
    \nabla \xi(a) = \begin{pmatrix}
2a_{11}+a_{22} & 0\\
0 & 2a_{22}+a_{11}
\end{pmatrix}
,\quad\forall a \in \R^{2\times 2},
\end{align*}
and thus, for every $a,b\in\R^{2\times2}$, 
\begin{align*}
    \frac{\d}{\d\eps}\big|_{\eps=0} \nabla \xi(a+\eps b)= \begin{pmatrix}
2b_{11}+b_{22} & 0\\
0 & 2b_{22}+b_{11}
\end{pmatrix}.
\end{align*}
For every nonzero $y\in\R^2$, if $b\in \S^2_+$ satisfies $y^\intercal by>0$, we must have $b_{11}>0$ or $b_{22}>0$, and due to the above display, we have that \eqref{e.ass.coupling} holds.

In fact, if we only keep the cross-term in \eqref{e.example}, so that the function $\xi$ is now given by $\xi(a)= a_{11} a_{22}$  and
\begin{align*}
    \frac{\d}{\d\eps}\big|_{\eps=0} \nabla \xi(a+\eps b)= \begin{pmatrix}
b_{22} & 0\\
0 & b_{11}
\end{pmatrix},
\end{align*}
then we see that $\xi$ is $e_1$-to-$e_2$ coupled and $e_2$-to-$e_1$ coupled, where $(e_1,e_2)$ is the canonical basis. Even though this case does not satisfy the assumptions of Theorem~\ref{t.main}, we see that we can still ensure simultaneous replica-symmetry breaking by applying Proposition~\ref{p.sim_RSB}. A similar phenomenon is also valid with more types, as discussed in the second paragraph after the statement of Theorem~\ref{t.main}.  
\end{remark}

\section{The case of multi-species models}

In this section, we describe the necessary adjustments for obtaining a version of Theorem~\ref{t.main} in the context of multi-species models. 

We fix a finite set $\sS$ containing symbols for different species.
We start by describing the multi-species model, which is the same as that in~\cite{multi-sp} with the parameter $\kappa_\s$ therein set to be $1$ for every $\s\in\sS$.

For each $N\in\N$, let $(I_{N,\s})_{\s\in\sS}$ be a partition of $\{1,\dots,N\}$. We interpret each $I_{N,\s}$ as the set of indices for spins belonging to the $\s$-species. 
For each $N\in\N$, we set
\begin{align}\label{e.lambda_N,s=}
    \lambda_{N,\s} = |I_{N,\s}|/N,\quad\forall \s\in\sS;\qquad \lambda_N=\Ll(\lambda_{N,\s}\Rr)_{\s\in\sS}.
\end{align}
We interpret $\lambda_{N,\s}$ as the fraction of the $\s$-species in terms of population.
For each $N \in \N\cup\{\infty\}$, we consider the simplex
\begin{align*}\blacktriangle_N = \Big\{(\lambda_\s)_{\s\in\sS}\ \big|\ \lambda_\s \in [0,1]\cap (\Z/N),\,\forall \s\in\sS;\; \sum_{\s\in\sS}\lambda_\s =1\Big\}
\end{align*}
with the understanding that $\Z/\infty = \R$ when $N=\infty$.
In view of~\eqref{e.lambda_N,s=}, we have $\lambda_N\in \blacktriangle_N$ for each $N\in\N$.

For each $\s\in\sS$, let $\mu_\s$ be a finite positive measure supported on $[-1,+1]$.
For every $N\in\N$, a spin configuration is of the form $\sigma = (\sigma_{1},\dots , \sigma_{N}) \in [-1,+1]^N$, where each spin~$\sigma_n$ is independently drawn from $\mu_\s$ if $n\in I_{N,\s}$.
In other words, denoting by $P_{N,\lambda_N}$ the distribution of $\sigma$, we have
\begin{align*}\d P_{N,\lambda_N}(\sigma) = \otimes_{\s\in\sS}\otimes_{n \in I_{N,\s}} \d\mu_\s(\sigma_{n}).
\end{align*}

For two spin configurations $\sigma,\sigma'$ of size $N$ and $\s\in\sS$, we consider the overlap of the $\s$-species:
\begin{align*}R_{N,\lambda_N,\s}(\sigma,\sigma') = \frac{1}{N}\sigma_{I_{N,\s}} \cdot\sigma'_{I_{N,\s}} \in [-1,+1]
\end{align*}
where
\begin{align}\label{e.sigma_bullet,I}
    \sigma_{I_{N,\s}} = (\sigma_{n})_{n\in I_{N,\s}}
\end{align}
is a vector in $\R^{|I_{N,\s}|}$ and similarly for $\sigma'_{I_{N,\s}}$. The entire overlap structure of the spin configurations is captured by the $\R^\sS$-valued overlap:
\begin{align*}
    R_{N,\lambda_N}(\sigma,\sigma') = \Ll(R_{N,\lambda_N,\s}(\sigma,\sigma')\Rr)_{\s\in\sS}.
\end{align*}

Let $\xi:\R^\sS\to\R$ be a smooth function and assume the existence of a centered Gaussian process $(H_N(\sigma))_{\sigma\in [-1,+1]^N}$ with covariance
\begin{align*}\E\Ll[ H_N(\sigma)H_N(\sigma')\Rr] = N\xi\Ll(R_{N,\lambda_N}(\sigma,\sigma')\Rr).
\end{align*}

Let $\square$ be a placeholder for subscripts. In the previous sections, we have used the notation $\mcl Q_\square$ for right-continuous increasing paths in $\S^\D_+$ with properties indicated by $\square$. Now we display the dependence on $\D$ by writing $\mcl Q_\square (\D)$.
We introduce the collection of paths appearing in the multi-species setting:
\begin{align*}\mcl Q^\sS_\square = \prod_{\s\in\sS} \mcl Q_\square(1).
\end{align*}
These paths can be thought of as the diagonal part of the paths that appeared in the previous sections. 
We now construct the external field parametrized by $q =(q_\s)_{\s\in\sS} \in \mcl Q^\sS_\infty$. Here, each $q_\s$ is a right-continuous increasing path in $\R_+$.
For each $\s\in\sS$ and $n\in I_{N,\s}$, let $(w^{q_\s}_n(\brho))_{\brho\in\supp\fR}$ be the real-valued centered Gaussian process given as in~\eqref{e.E[ww]=} (for $\D=1$) with covariance 
\begin{align}\label{e.Ew^q_s_iw^q_s_i=}
    \E\Ll[ w^{q_\s}_n (\brho)w^{q_\s}_n(\brho')\Rr]  = q_\s(\brho\wedge\brho).
\end{align}
Conditioned on $\fR$, we assume that all these processes, indexed by $\s$ and $n$, are independent. For each $\s$, we write $w^{q_\s}_{I_{N,\s}} = \Ll(w^{q_\s}_n\Rr)_{n\in I_{N,\s}}$. Recall the notation in~\eqref{e.sigma_bullet,I}. For each $N\in\N$ and $q\in \mcl Q^\sS_\infty$, we define
\begin{align*}W^q_N(\sigma,\brho) = \sum_{s\in \sS}w^{q_\s}_{I_{N,\s}}(\brho)\cdot \sigma_{I_{N,\s}}
\end{align*}
which, conditioned on $\fR$, is a centered Gaussian process with covariance
\begin{align*}\E \Ll[ W^q_N(\sigma,\brho) W^q_N(\sigma',\brho')\Rr]\stackrel{\eqref{e.Ew^q_s_iw^q_s_i=}}{=} N q(\brho\wedge\brho')\cdot R_{N,\lambda_N}(\sigma,\sigma').
\end{align*}

Now, for $N\in \N$, $\lambda_N\in \blacktriangle_N$, $t\in[0,\infty)$, and $q \in \mcl Q_\infty^\sS$, we consider the Hamiltonian
\begin{align*}H^{t,q}_N(\sigma,\brho)= \sqrt{2t}H_N(\sigma) - t N \xi \Ll(R_{N,\lambda_N}(\sigma,\sigma)\Rr)  
    + \sqrt{2}W^q_N(\sigma,\brho) - Nq(1)\cdot R_{N,\lambda_N}(\sigma,\sigma),
\end{align*}
where $q(1)= (q_\s(1))_{\s\in\sS}\in \R^\sS_+$ and $q(1)\cdot R_{N,\lambda_N}(\sigma,\sigma) = \sum_{s\in\mathscr{\sS}}q_\s(1)\cdot R_{N,\lambda_N,\s}(\sigma,\sigma)$. We define the associated free energy and Gibbs measure
\begin{gather*}
    \bar F_{N,\lambda_N}(t,q) = - \frac{1}{N}\E\log \iint  \exp\Ll( H^{t,q}_N(\sigma,\brho)\Rr)\d P_{N,\lambda_N}(\sigma)\d \fR(\brho),  \\
    \la\cdot\ra_{N,\lambda_N} \propto \exp\Ll( H^{t,q}_N(\sigma,\brho)\Rr)\d P_{N,\lambda_N}(\sigma)\d \fR(\brho). \end{gather*}
Here, $\E$ first averages over all the Gaussian randomness in $H_N(\sigma)$ and $W^q_N(\sigma,\brho)$ and then the randomness in $\fR$. 
The dependence of $F_{N,\lambda_N}(t,q)$ on the partition $(I_{N,\s})_{\s\in\sS}$ is only through $\lambda_N$, which is the reason for us to set the notation in this way.

For each $q=(q_\s)_{\s\in\sS} \in \mcl Q^\sS_\infty$ and $\lambda_\infty = (\lambda_{\infty,s})_{\s\in\sS}\in \blacktriangle_\infty$, we define
\begin{gather*}
    \psi_{\mu_\s}(q_\s)=-\E \log \iint\exp\Ll(\sqrt{2}w^{q_\s}(\brho)\cdot \tau- q_\s(1)\tau^2\Rr)\d \mu_\s(\tau)\d\fR(\brho),\quad\forall s \in\sS;\\
    \psi_{\lambda_\infty}(q) = \sum_{\s\in\sS}\lambda_{\infty,\s}\psi_{\mu_\s}(q_\s).\end{gather*}
Here, $(w^{q_\s}(\brho))_{\brho\in\supp\fR}$ is the real-valued process given as in~\eqref{e.E[ww]=} (for $\D=1$).

\begin{remark}\label{r.psi_mu_s}
It is important to notice that $\psi_{\mu_\s}(q_\s)$ is exactly $\psi(q)$ given in~\eqref{e.psi=} with $\D=1$ (set $\mcl Q_\infty$ to be $\mcl Q_\infty(1)$, $q=q_\s$ and $P_1=\mu_\s$).
Hence, results from previous sections are also valid for $\psi_{\mu_\s}$. \qed
\end{remark}

For each $\s\in\sS$, the derivative $\partial_{q_\s} \psi_{\mu_\s}(q_\s)\in \mcl Q_\infty(1)$ is defined in the same way as that for $\psi$ in~\eqref{e.psi=}. By~\cite[(4.18) and Lemma~4.9]{multi-sp}, $\psi_{\lambda_\infty}$ is differentiable at every $q\in \mcl Q_2^\sS$ and its derivative $\partial_q\psi_{\lambda_\infty}(q)$ satisfies
\begin{align}\label{e.d_qpsi_lambda=}
    \partial_q \psi_{\lambda_\infty}(q)= \Ll(\lambda_{\infty, \s} \, \partial_{q_\s} \psi_{\mu_\s}(q_\s)\Rr)_{\s\in\sS} \in \mcl Q^\sS_\infty.
\end{align}

For every $\lambda_\infty\in \blacktriangle_\infty$ and $(t,q)\in\R_+\times \mcl Q^\sS_2$, we consider the functional
\begin{align*}\mcl J_{\lambda_\infty,t,q}(q',p) = \psi_{\lambda_\infty}(q')+\int_0^1 p\cdot(q-q')+ t\int_0^1\xi(p)
\end{align*}
defined for every $q'\in Q^\sS_2$ and $p\in L^2([0,1],\R^\sS)$. Similarly to the vector case, we say that $(q',p) \in Q^\sS_2 \times L^2([0,1],\R^\sS)$ is a critical point of $\mcl J_{\lambda_\infty,t,q}$ if
\begin{equation*}  q = q' - t \nabla \xi(p) \quad \text{ and } \quad p = \partial_q \psi_{\lambda_\infty}(q'),
\end{equation*}
where we write $\nabla\xi = (\partial_\s \xi)_{\s\in\sS}$. 

Similarly to the vector case, we have that up to a small perturbation of the energy function and up to the extraction of a subsequence in $N$, and for $\sigma, \sigma'$ two independent samples from the Gibbs measure, the overlap $R_{N,\lambda_N, \s}(\sigma,\sigma')$ converges in law to $p(U)$, where $(q',p)$ is a critical point of $\mcl J_{\lambda_\infty,t,q}$ and $U$ is a uniform random variable on $[0,1]$. We refer to  \cite[Theorem~1.4]{multi-sp} for the precise statement.

For every $a,b\in \R^\sS$, we write $a\geq b$ if $a_\s\geq b_\s$ for every $\s\in\sS$. In the current multi-species setting, for $\s,\s'\in\sS$, we say that $\xi$ is \textbf{$\s$-to-$\s'$ coupled} provided that, for every $a,b\in\R^\sS_+$, we have
\begin{align*}a\geq b,\ a_\s>b_\s\qquad\Longrightarrow \qquad \partial_{\s'}\xi(a)>\partial_{\s'}\xi(b).
\end{align*}
\begin{theorem}[Simultaneous RSB in multi-species models]\label{t.sim_RSB_MS}
Let $p,q \in \mcl Q^\sS_\infty$ satisfy $p=\partial_q \psi_{\lambda_\infty}(q+t\nabla\xi(p))$ for some $\lambda_\infty\in\blacktriangle_\infty$, $t > 0$, and $\xi$.
Suppose that $\xi$ is $\s$-to-$\s'$ coupled for some $\s,\s'\in\sS$ and that $\mu_\s$ is not a Dirac mass at $0$. For every $0 < r < r' < 1$, if ${p_\s}(r')>{p_\s}(r)$, then ${p_{\s'}}(r')> {p_{\s'}}(r)$.
\end{theorem}

\begin{proof}
Remark~\ref{r.psi_mu_s} allows us to apply the results from Section~\ref{s.detect}, stated for more general $\psi$ as in~\eqref{e.psi=}, to $\psi_{\mu_\s}$ here.
Using~\eqref{e.d_qpsi_lambda=}, we obtain $p_{\s'} = \lambda_{\infty,\s'}\partial_{q_{\s'}}\psi_{\mu_{\s'}}(q_{\s'}+t\partial_{\s'}\xi(p))$ from $p=\partial_q \psi_{\lambda_\infty}(q+t\nabla\xi(p))$.
Assuming that $\overrightarrow{p_\s}(r')>\overrightarrow{p_\s}(r)$ and using that $\xi$ is $\s$-to-$\s'$ coupled, we can apply Lemma~\ref{l.detect_gen_incre} with $p,q,\psi,P_1,D$ therein replaced by $\lambda_{\infty,\s'}^{-1}p_{\s'}$, $q_{\s'}+t\partial_{\s'}\xi(p)$, $\psi_{\mu_\s}$, $\mu_\s$, $1$ respectively to obtain that $\overrightarrow{p_{\s'}}(r') > \overrightarrow{p_{\s'}}(r)$.

This shows the announced result with $p$ replaced by its left-continuous version $\ap$. As in the proof of Theorem~\ref{t.main}, we can also obtain the statement with the right-continuous path $p$ by examining the case of a jump, i.e.\ we argue that if for some $r \in (0,1)$, we have $\overrightarrow{p_\s}(r+) > \overrightarrow{p_{\s}}(r)$, then $\overrightarrow{p_{\s'}}(r+) > \overrightarrow{p_{\s'}}(r)$. Indeed, this follows from Corollary~\ref{c.jump_iff}.
\end{proof}

\smallskip

\noindent \textbf{Funding.} HBC is funded by the Simons Foundation. 

\noindent
\textbf{Data availability.}
No datasets were generated during this work.

\noindent
\textbf{Conflict of interests.}
The authors have no conflicts of interest to declare.

\noindent
\textbf{Competing interests.}
The authors have no competing interests to declare.

\small
\bibliographystyle{plain}
\newcommand{\noop}[1]{} \def\cprime{$'$}

\end{document}